\documentclass[11pt,amsfonts]{article}
\usepackage[margin=1in]{geometry}  
\usepackage{graphicx}              

\usepackage{amsmath, amssymb, amsthm, epsfig}               
\usepackage{enumerate}
\usepackage{amsfonts}              
\usepackage{amsthm}                
\usepackage{amsthm}
\usepackage{enumerate}
\theoremstyle{plain}
\usepackage{color}
\newtheorem{corollary}{Corollary}[section]
\newtheorem{definition}[corollary]{Definition}

\newtheorem{lemma}[corollary]{Lemma}
\newtheorem{prp}[corollary]{Proposition}
\newtheorem{remark}[corollary]{Remark}
\newtheorem{thm}[corollary]{Theorem}

\newfont{\sBlackboard}{msbm10 scaled 900}

\newcommand{\mylabel}[1]{\label{#1}
	\ifx\undefined\stillediting
	\else \fbox{$#1$}\fi }
\newcommand{\BE}{\begin{equation}}

\newcommand{\EEQ}{\end{equation}}
\newcommand{\rfb}[1]{\mbox{\rm
		(\ref{#1})}\ifx\undefined\stillediting\else:\fbox{$#1$}\fi}

\newfont{\Blackboard}{msbm10 scaled 1200}

\newfont{\roma}{cmr10 scaled 1200}

\newcommand{\bb}{\begin{equation}}
\newcommand{\bbb}{\end{equation}}

\DeclareMathOperator{\meas}{meas}

\newcommand{\mm}    {{\hbox{\hskip 0.5pt}}}

\newcommand{\bluff} {{\hbox{\raise 15pt \hbox{\mm}}}}
scaled\magstep2

\makeatletter
\def\section{\@startsection {section}{1}{\z@}{-3.5ex plus -1ex minus
		-.2ex}{2.3ex plus .2ex}{\large\bf}}
\makeatother

\begin{document}
\title{Embedding theorems in the fractional Orlicz-Sobolev space and applications to non-local problems}
\author{Sabri Bahrouni and Hichem Ounaies}

\maketitle

\begin{abstract}
In the present paper, we deal with a new continuous and compact embedding theorems for the fractional Orlicz-Sobolev spaces, also, we study the existence of infinitely many nontrivial solutions for a class of non-local fractional Orlicz-Sobolev Schr\"{o}dinger equations whose simplest prototype is
$$(-\triangle)^{s}_{m}u+V(x)m(u)=f(x,u),\ x\in\mathbb{R}^{d},$$
where $0<s<1$, $d\geq2$ and $(-\triangle)^{s}_{m}$ is the fractional $M$-Laplace operator.
The proof is based on the variant Fountain theorem established by Zou. 
\end{abstract}

{\small \textbf{Keywords:} Fractional Orlicz-Sobolev space, Compact embedding theorem, Fractional $M-$Laplacian, Fountain Theorem.

\section{Introduction and main result}

In this paper, we are concerned with the study of the fractional $M$-Laplacian equation:
\begin{equation}\label{eq1}
(-\triangle)^{s}_{m}u+V(x)m(u)= f(x,u),\ x\in\mathbb{R}^{d}, \\
\end{equation}
where $(-\triangle)^{s}_{m}$ is the fractional $M$-Laplace operator,  $0<s<1$, $d\geq2$, $m:\mathbb{R}\rightarrow\mathbb{R}$ is an increasing homeomorphism, $V:\mathbb{R}^{d}\rightarrow\mathbb{R}$ and $f:\mathbb{R}^{d}\times \mathbb{R}\rightarrow\mathbb{R}$ are given functions.\\

In the last years, problem \eqref{eq1} has received a special attention for the case $s=1$ and $m(t)=t$, that is, when it is of the form
\begin{equation}\label{equa}
  -\triangle u+V(x)u= f(x,u),\ x\in\mathbb{R}^{d}.
\end{equation}
We do not intend to review the huge bibliography related to the equations like \eqref{equa}, we just emphasize that the potential $V:\mathbb{R}^{d}\rightarrow\mathbb{R}$ has a crucial role concerning the existence and behaviour of solutions. For example, when $V$ is radially symmetric, it is natural to look for radially symmetric solutions, see \cite{Strauss, Willem}. On the other hand, after the paper of Rabinowitz \cite{Rabinowitz} where the potential $V$ is assumed to be coercive, several different assumptions are adopted in order to obtain existence and multiplicity results (see \cite{Anouar, Bartch, Rad5, Zhang, ZXu}).

For the case $s=1$, problem \eqref{eq1} becomes
$$-\triangle_{m}u+V(x)m(u)= f(x,u),\ x\in\mathbb{R}^{d},$$ where the operator $\triangle_{m}u=\text{div}(m(|\nabla u|)|\nabla u|)$ named $M$-Laplacian.
This class of problems arises in a lot of applications, such as, {\it Nonlinear Elasticity, Plasticity, Generalized Newtonian Fluid, Non-Newtonian Fluid, Plasma Physics}. The reader can find more details involving this subject in \cite{Alves, Rad, Rad1, Rad2} and the references therein.\\

Notice that when $0<s<1$ and $m(t)=|t|^{p-2}t$, $p\geq2$, problem \eqref{eq1} gives back the fractional Schr\"{o}dinger equation
\begin{equation}\label{eq2}
  (-\triangle)^{s}_{p}u+V(x)|u|^{p-2}u= f(x,u),\ x\in\mathbb{R}^{d} ,
\end{equation}
where $ (-\triangle)^{s}_{p}$ is the non-local fractional $p$-Laplacian operator. The literature on non-local operators and on their applications is quite large. We can quote \cite{4,14,23,25,26} and the references therein.  We also refer to the recent monographs \cite{14,22} for a thorough variational approach of  non-local problems.
In the last decade, many several existence and multiplicity results have been obtained concerning the equation \eqref{eq2}, (see \cite{pucci, chang, Torres}).
In \cite{Rad 2}, the authors studied the existence of multiple solutions where the nonlinear term $f$ is assumed to have a superlinear behaviour at the origin and a sublinear decay at infinity.
In \cite{Ambrosio}, Vincenzo studied the existence of infinitely many solutions for the problem \eqref{eq2}, when $f$ is superlinear and $V$ can change sign.\\

Contrary to the classical fractional Laplacian Schr\"{o}dinger equation that is widely investigated, the situation seems to be in a developing state when the new fractional $M$-Laplacian is present. In this context, the natural setting for studying problem \eqref{eq1} are fractional Orlicz-Sobolev spaces.
Currently, as far as we know, the only results for fractional Orlicz-Sobolev spaces and fractional $M$-Laplacian operator are obtained in \cite{Azroul, Sabri,7, B, B1, Napoli, Salort}. In particular, in \cite{7}, Bonder and Salort define the fractional Orlicz-Sobolev space associated to an $N$-function $M$ and a fractional parameter $0<s<1$ as
$$W^{s,M}(\Omega)=\bigg{\{}u\in L^{M}(\Omega):\ \int_{\Omega}\int_{\Omega}M\bigg{(}\frac{u(x)-u(y)}{|x-y|^{s}}\bigg{)}\frac{dxdy}{|x-y|^{d}}<\infty\bigg{\}},$$
where $\Omega$ is an open subset of $\mathbb{R}^{d}$ and $L^{M}(\Omega)$ is the Orlicz space. They also define the fractional $M$-Laplacian operator as,
\begin{equation}\label{newoperator}
(-\triangle)^{s}_{m}u(x)=P.V.\int_{\mathbb{R}^{d}} m\bigg{(}\frac{u(x)-u(y)}{|x-y|^{s}}\bigg{)}\frac{dy}{|x-y|^{d+s}},
\end{equation}
this operator is a generalization of the fractional $p$-Laplacian. 
 They established compact embedding result referred to under the blanket title "the Fr\`{e}chet-Kolmogorov compactness theorem" witch gives the compact embedding of $W^{s,M}(\Omega)\hookrightarrow L^{M}(\Omega)$ when $\Omega$ is a bounded in $\mathbb{R}^{d}$. They also deduce some consequences such as $\Gamma$-convergence of the modulars and convergence of solutions for some fractional versions of the $(\triangle)^{s}_{m}$ operator as the fractional parameter $s\uparrow1$.\\

Motivated by these above results,
 our first aim is to prove the compact embedding $W^{s,M}(\Omega)\hookrightarrow L^{M_{*}}(\Omega)$ where $M_{*}$ is the Sobolev conjugate of $M$ and $\Omega$ is bounded. Furthermore, we state the continuous embedding $W^{s,M}(\mathbb{R}^{d})\hookrightarrow L^{M_{*}}(\mathbb{R}^{d})$. Hence the compact embedding $W^{s,M}(\Omega)\hookrightarrow L^{\Phi}(\Omega)$ and the continuous embedding $W^{s,M}(\mathbb{R}^{d})\hookrightarrow L^{\Phi}(\mathbb{R}^{d})$ remain true for any $N$-function $\Phi$ such that $M_*$ is essentially stronger than $\Phi$. (see Definition \ref{prec}).\\ 

  Our next aim is to study the existence and the multiplicity of nontrivial weak solutions of problem \eqref{eq1}, where the new fractional $M$-Laplacian is present. Under suitable conditions on the potentials $V$ and $f$ (will be fixed bellow), we deal with a new compact embedding theorem on the whole space $\mathbb{R}^{d}$. Also we establish some useful inequalities  which yields to apply a variant of Fountain theorem due to Zou \cite{Zou}. As far as we know, all these results are new.\\

 Related to functions $m,M,V$ and $f$, our hypotheses are the following:\\

 \textbf{Conditions on} $m$ and $M$:\\

 \begin{enumerate}
  \item [($m_{1}$)] 
  $ 1<m_{0}=\displaystyle \inf_{t>0}\frac{tm(t)}{M(t)}\leq \frac{tm(t)}{M(t)}\leq m^{0}=\displaystyle \sup_{t>0}\frac{tm(t)}{M(t)}<m_{0}^{*}<\infty,\ \text{for all}\ t\neq0$ where $$M(t)=\int_{0}^{|t|}m(s)ds\ \ \text{and}\ \ m_{0}^{*}=\frac{dm_{0}}{d-m_{0}}.$$

\end{enumerate}

\noindent $(M_{1})$ There exists $1<\mu<m_{0}$, such that
             $$\displaystyle\lim_{|t|\rightarrow+\infty}\displaystyle\frac{|t|^{\mu}}{M(t)}=0.$$

\noindent $(M_{2})$ The function  $t\mapsto M(\sqrt{t}),\ t\in[0,\infty[$ is convex.\\

\noindent $(M_{3})$ $\displaystyle\int_{0}^{1}\frac{M^{-1}(\tau)}{\tau^{\frac{d+s}{d}}}d\tau<\infty$ and
$\displaystyle\int_{1}^{+\infty}\frac{M^{-1}(\tau)}{\tau^{\frac{d+s}{d}}}d\tau=\infty$, where $0<s<1$.\\

 \textbf{Conditions on} $V$:\\

\noindent $(V_{1})$ $V\in C(\mathbb{R}^{d},\mathbb{R})$ and $\inf_{x\in\mathbb{R}^{d}}V(x)\geq V_{0}>0$.\\

\noindent $(V_{2})$ $meas(\{x\in\mathbb{R}^{d}:\ V(x)\leq L\})<\infty,\ \text{for all}\ L>0,$
where $meas(.)$ denotes the Lebesgue measure in $\mathbb{R}^{d}$.\\

\textbf{Conditions on} $f$:\\

 \noindent $(f_{1})$ $f(x,u)=p\xi(x)|u|^{p-2}u$, $1<p<\mu$ and $\xi:\ \mathbb{R}^{d}\rightarrow\mathbb{R}$ is a positive continuous function such that $\xi\in L^{\frac{\mu}{\mu-p}}(\mathbb{R}^{d})$.

\begin{remark}
  We mention some examples of functions $m$ which are increasing homeomorphisms and satisfy conditions $(m_{1})$, $(M_{1})-(M_{2})$:

  \begin{enumerate}
    \item $m(t)= q|t|^{q-2}t$, for all $t\in\mathbb{R}$, with $2<q<d$ (also satisfies condition $(M_{3})$).
    \item $m(t)= p|t|^{p-2}t+ q|t|^{q-2}t$, for all $t\in\mathbb{R}$, with $2<p<q<d$.
    \item $m(t)= q|t|^{q-2}t\log(1+|t|)+\displaystyle\frac{|t|^{q-1}t}{1+|t|}$,  for all $t\in\mathbb{R}$, with $2<q<d$.
  \end{enumerate}
\end{remark}

Under the above hypotheses, we state our main results.

\begin{thm}\label{ceb}
	Let $M$ be an $N$-function and $s\in(0,1)$. Let $\Omega$ be a bounded open subset of $\mathbb{R}^{d}$ with $C^{0,1}$-regularity and bounded boundary.
\begin{enumerate}
  \item If $(m_{1})$ and $(M_{3})$ hold, then the embedding
	\begin{equation}\label{7}
	W^{s,M}(\Omega)\hookrightarrow L^{M_{*}}(\Omega),
	\end{equation}
is continuous.
  \item Moreover, for any $N$-function $B$ such that $M_{*}$ is essentially stronger than $B$, denoted $B\prec\prec M_{*}$ (see Definition \ref{prec}), the embedding
   \begin{equation}\label{Bem}
	W^{s,M}(\Omega)\hookrightarrow L^{B}(\Omega),
	\end{equation}
is compact.
\end{enumerate}
\end{thm}

$\,$

The boundedness of $\Omega$ in Theorem \ref{ceb} is a natural requirement for the compactness theorem, but, as we shall show in the next theorem, not necessary for the continuous embedding.

\begin{thm}\label{allR}
Let $M$ be an $N$-function and $s\in(0,1)$.
\begin{enumerate}
  \item If $(m_{1})$ and $(M_{3})$ hold, then the embedding
  $$W^{s,M}(\mathbb{R}^{d})\hookrightarrow L^{M_{*}}(\mathbb{R}^{d}),$$ is continuous.
  \item Moreover, for any $N$-function $B$ such that $B\prec\prec M_{*}$, the embedding
	$$W^{s,M}(\mathbb{R}^{d})\hookrightarrow L^{B}(\mathbb{R}^{d}),$$ is continuous.
\end{enumerate}

\end{thm}

In studying the existence of solution of problem \eqref{eq1}, it is common to relax the notion of solution by considering weak solutions. By these we understand functions in $W^{s,M}(\Omega)$ that satisfy \eqref{eq1} in sense of distribution.
\begin{thm}\label{thm1}
  Suppose that $(m_{1})$, $(M_{1})-(M_{2})$, $(V_{1})-(V_{2})$ and $(f_{1})$ hold. Then, problem \eqref{eq1} possesses infinitely many nontrivial weak solutions.
\end{thm}

This paper is organized as follows. In Section $2$, we give some
definitions and fundamental properties of the spaces
$L^{M}(\Omega)$ and $W^{s,M}(\Omega)$. In Section $3$,  we prove Theorems \ref{ceb} and \ref{allR}. In Section $4$, we introduce our abstract framework related to problem \eqref{eq1}. Finally, in Section $5$, using a variant Fountain theorem \cite{Zou}, we prove Theorem \ref{thm1}.

\section{Preliminaries}

We start by recalling some basic facts about Orlicz spaces $L^{M}(\Omega)$. For more details we refer to the books by Adams \cite{Adams}, Kufner {\it et al.} \cite{Kufner}, Rao and Ren \cite{Rao} and the papers by Cl\'{e}ment {\it et al.} \cite{clement1, clement}, Fukagai {\it et al.} \cite{Fukagai}, Garc\'{\i}a-Huidobro {\it et al.} \cite{Garcia} and Gossez \cite{Gossez}.

\subsection{Orlicz spaces}


\noindent $\bullet$ $B^{c}_{R}(0)=\mathbb{R}^{d}\setminus B_{R}(0)$.\\
$\bullet$ $\|u\|_{L^{\mu}(\mathbb{R}^{d})}=\bigg{(}\displaystyle\int_{\mathbb{R}^{d}}|u(x)|^{\mu}dx\bigg{)}^{\frac{1}{\mu}}$.\\

 Let $M: \mathbb{R}\rightarrow\mathbb{R}_{+}$ be an $N$-function, i.e,

\begin{enumerate}
	\item $M$ is even, continuous, convex,
	\item $\frac{M(t)}{t}\rightarrow 0$ as $t\rightarrow0$ and  $\frac{M(t)}{t}\rightarrow +\infty$ as $t\rightarrow+\infty$.
\end{enumerate}

Equivalently, $M$ admits the representation: $$M(t)=\int_{0}^{|t|}m(s)ds,$$ where $m: \mathbb{R}\rightarrow\mathbb{R}$ is non-decreasing, right continuous, with $m(0)=0$, $m(t)>0\ \text{for all}\ t>0$ and $m(t)\rightarrow\infty$ as $t\rightarrow\infty$ (see \cite{Krasnoselskii}, page 9). We call the conjugate function of $M$, the function denoted $\overline{M}$ and defined by
$$\overline{M}(t)=\int_{0}^{|t|}\overline{m}(s)ds,$$ where  $\overline{m}: \mathbb{R}\rightarrow\mathbb{R}$,
$\overline{m}(t)=\sup\{s:\ m(s)\leq t\}$.
We observe that $\overline{M}$ is also an $N$-function and the following Young's inequality holds true

\begin{equation}\label{Young}
st\leq M(s)+\overline{M}(t)\ \ \text{for all}\ s,t\geq0,\ \ (\text{see}\ \cite{Adams},\ \text{page}\ 229).
\end{equation}
Equality holds in \eqref{Young} if and only if either $t=m(s)$ or $s=\overline{m}(t)$.\\

If $(M_{3})$ is satisfied, another important function related to function $M$, it is the Sobolev conjugate $N$-function $M_{*}$ of $M$ defined by,

\begin{equation}\label{6}
M_{*}^{-1}(t)=\int_{0}^{t}\frac{M^{-1}(\tau)}{\tau^{\frac{d+s}{d}}}d\tau.
\end{equation}

In what follows, we say that an $N$-function $M$ satisfies the $\triangle_{2}$-condition, if
\begin{equation}\label{delta2}
  M(2t)\leq K\ M(t)\ \text{for all}\ t\geq0,
\end{equation}
for some constant $K>0$. This condition can be rewritten in the following way\\ For each $s>0$, there exists $K_{s}>0$ such that
\begin{equation}\label{deltas}
  M(st)\leq K_{s}\ M(t),\ \text{for all}\ t\geq0,\ \ (\text{see}\ \cite{Krasnoselskii},\ \text{page}\ 23).
\end{equation}

\begin{definition}\label{prec}
  Let $A$ and $B$ be two $N$-functions, we say that $A$ is essentially stronger than $B$, $B\prec\prec A$ in symbols, if for each $a>0$ there exists $x_{a}\geq0$ such that $$B(x)\leq A(ax),\ x\geq x_{a}.$$
\end{definition}

The  previous definition \ref{prec} is equivalent to, $$\lim_{t\rightarrow+\infty}\frac{B(kt)}{A(t)}=0,\ \text{for all positive constant}\ k \  (\text{see}\ \cite{Rao},\ \text{Theorem}\ 2).$$
Let $\Omega$ be an open subset of $\mathbb{R}^{d}$ and $\rho(u,M)=\displaystyle\int_{\Omega}M(u(x))dx$. The Orlicz space $L^{M}(\Omega)$ is the set of equivalence classes of real-valued measurable functions $u$ on $\Omega$  such that $\rho(\lambda u,M)<\infty\ \text{for some}\ \lambda>0.$\\

$L^{M}(\Omega)$ is a Banach space under the Luxemburg norm
\begin{equation}\label{19}
\|u\|_{(M,\Omega)}=\inf\bigg{\{}\lambda>0\ :\ \int_{\Omega}M(\frac{u}{\lambda})dx\leq1\bigg{\}},
\end{equation}
 if there is no confusion we shall write $\|.\|_{(M)}$ instead of $\|.\|_{(M,\Omega)}$, whose norm is equivalent to the Orlicz norm $$\|u\|_{L^{M}(\Omega)}=\sup_{\rho(v,\overline{M})\leq1}\int_{\Omega}|u(x)||v(x)|dx,$$ and for each $u\in L^{M}(\Omega)$,
 \begin{equation}\label{tman}
   \|u\|_{(M)}\leq\|u\|_{ L^{M}(\Omega)}\leq 2 \|u\|_{(M)}\ \ (\text{see}\cite{Kufner},\ \text{Theorem}\ 4.8.5).
 \end{equation}

 The $\triangle_{2}$-condition with $(M_{2})$ ensures that the Orlicz space $L^{M}(\Omega)$ is a uniformly convex space and thus, a reflexive Banach space (see \cite{Mihai}, Proposition $2.2$).\\

The Orlicz spaces H\"{o}lder's inequality reads as follows:  (see \cite{Kufner}, Theorem 4.7.5)
$$\int_{\Omega}|uv|dx\leq \|u\|_{L^{M}(\Omega)}\|v\|_{L^{\overline{M}}(\Omega)}\ \ \text{for all}\ \ u\in L^{M}(\Omega)\ \text{and}\ v\in L^{\overline{M}}(\Omega).$$

In the following, we recall a few results which will be useful in the sequel.

\begin{prp}[\cite{Adams}]\label{cv}
	Let $(u_{n})_{n\in\mathbb{N}}$ be a sequence in $L^{M}(\Omega)$ and $u\in L^{M}(\Omega)$.
	If $M$ satisfies the $\triangle_{2}$-condition and $\rho(u_{n}-u,M)\rightarrow0$, then $u_{n}\rightarrow u$ in $ L^{M}(\Omega)$ .
\end{prp}

\begin{prp}[\cite{Kufner}]\label{KML}
Let $M$ be an $N$-function. Then $$\|u\|_{L^{M}(\Omega)}\leq \rho(u,M)+1,\ \text{for all}\ u\in L^{M}(\Omega).$$
\end{prp}

 \begin{lemma}[\cite{clement}]\label{lem2}
   Let $G$ be an $N$-function satisfying $$1<g_{0}:=\inf_{t>0}\frac{t g(t)}{G(t)}\leq \frac{t g(t)}{G(t)}\leq g^{0}:=\sup_{t>0}\frac{t g(t)}{G(t)}<\infty$$
   where $g=G'$ and let $\xi_{0}(t)=\min\{t^{g_{0}},t^{g^{0}}\}$, $\xi_{1}(t)=\max\{t^{g_{0}},t^{g^{0}}\}$, for all $t\geq0$. Then
   $$\xi_{0}(\beta)G(t)\leq G(\beta t)\leq \xi_{1}(\beta)G(t)\ \text{for}\ \beta,t\geq0,$$ and $$\xi_{0}(\|u\|_{(G,\Omega)})\leq \int_{\Omega}G(|u|)dx\leq\xi_{1}(\|u\|_{(G,\Omega)})\ \text{for}\ u\in L^{G}(\Omega).$$
 \end{lemma}

 \begin{lemma}[\cite{Fukagai}]\label{M***}
  Let $M$ be an $N$-function satisfying $(m_{1})$ and $(M_{3})$, then the function $M_{*}$ satisfies the following inequality
   $$m_{0}^{*}=\frac{dm_{0}}{d-m_{0}}\leq \frac{t m_{*}(t)}{M_{*}(t)}\leq (m^{0})^{*}=\frac{dm^{0}}{d-m^{0}},$$ where $m_{*}$ is such that $M_{*}(t)=\int_{0}^{|t|}m_{*}(s)ds.$
 \end{lemma}

\subsection{Fractional Orlicz-Sobolev spaces}

In this subsection we give a brief overview on the fractional Orlicz-Sobolev spaces studied in \cite{7}, and the associated fractional $M$-laplacian operator.

\begin{definition}
	Let $M$ be an $N$-function. For an open subset $\Omega$ in $\mathbb{R}^{d}$ and $0<s<1$, we define the fractional Orlicz-Sobolev space $W^{s,M}(\Omega)$ as follows,
	\begin{equation}\label{20}
	W^{s,M}(\Omega)=\bigg{\{}u\in
L^{M}(\Omega):\ \int_{\Omega}\int_{\Omega}M\bigg{(}\frac{u(x)-u(y)}{|x-y|^{s}}\bigg{)}\frac{dxdy}{|x-y|^{d}}<\infty\bigg{\}}.
	\end{equation}
This space is equipped with the norm,
	\begin{equation}\label{21}
	\|u\|_{(s,M,\Omega)}=\|u\|_{(M,\Omega)}+[u]_{(s,M,\Omega)},
	\end{equation}
	where $[.]_{(s,M,\Omega)}$ is the Gagliardo semi-norm, defined by
	\begin{equation}\label{22}
	[u]_{(s,M,\Omega)}=\inf\bigg{\{}\lambda>0:\ \int_{\Omega}\int_{\Omega} M\bigg{(}\frac{u(x)-u(y)}{\lambda|x-y|^{s}}\bigg{)}\frac{dxdy}{|x-y|^{d}}\leq1\bigg{\}},
	\end{equation}
if there is no confusion we shall write $[.]_{(s,M)}$ and $\|.\|_{(s,M)}$ instead of $[.]_{(s,M,\Omega)}$ and $\|.\|_{(s,M,\Omega)}$ respectively.
\end{definition}

\begin{prp}[\cite{7}]
  Let $M$ be an $N$-function such that $M$ and $\overline{M}$ satisfy the $\triangle_{2}$-condition, and consider $s\in(0,1)$. Then $W^{s,M}(\mathbb{R}^{d})$ is a reflexive and separable Banach space. Moreover, $C^{\infty}_{0}(\mathbb{R}^{d})$ is dense in $W^{s,M}(\mathbb{R}^{d})$.
\end{prp}

A variant of the well-known Fr\`{e}chet-Kolmogorov compactness theorem gives the compactness of the
embedding of $W^{s,M}$ into $L^{M}$.

\begin{thm}[\cite{7}]\label{th}
  Let $M$ be an $N$-function, $s\in(0,1)$ and $\Omega$ a bounded open set in $\mathbb{R}^{d}$. Then the embedding $$W^{s,M}(\Omega)\hookrightarrow L^{M}(\Omega)$$ is compact.
\end{thm}

We recall that the fractional $M$-Laplacian operator is defined as
\begin{equation}\label{17}
(-\triangle)^{s}_{m}u(x)= P.V.\int_{\mathbb{R}^{d}} m\bigg{(}\frac{u(x)-u(y)}{|x-y|^{s}}\bigg{)}\frac{dy}{|x-y|^{d+s}},
\end{equation}
where $P.V$ is the principal value.\\
This operator is well defined between $W^{s,M}(\mathbb{R}^{d})$ and its dual space  $W^{-s,\overline{M}}(\mathbb{R}^{d})$. In fact, in [\cite{7}, Theorem 6.12] the following representation formula is provided
\begin{equation}\label{18}
\langle(-\triangle)^{s}_{m}u,v\rangle=\frac{1}{2}\int_{\mathbb{R}^{d}}\int_{\mathbb{R}^{d}} m\bigg{(}\frac{u(x)-u(y)}{|x-y|^{s}}\bigg{)}\frac{v(x)-v(y)}{|x-y|^{s}}\frac{dxdy}{|x-y|^{d}},
\end{equation}
for all $v\in W^{s,M}(\mathbb{R}^{d})$.

\section{Embedding Theorems}
After the above brief review, we are able to prove our main results involving
the fractional Orlicz-Sobolev spaces. 

\subsection{Proof of Theorem \ref{ceb}}

The proof will be carried out in several lemmas. we start by establishing an estimate for the Sobolev conjugate $N$-function $M_{*}$ defined by \eqref{6}.

\begin{lemma}\label{ceb1}
  Let $M$ be an $N$-function satisfying $(m_{1})$, $(M_{3})$ and $s\in(0,1)$. 
  Then the following conclusions hold true.
  \begin{enumerate}
    \item $t\mapsto[M_{*}(t)]^{\frac{d-s}{d}}$ is an $N$-function.
    \item For every $\epsilon>0$, there exists a constant $K_{\epsilon}$ such that for every $t$,
       \begin{equation}\label{Kepsilon}
         [M_{*}(t)]^{\frac{d-s}{d}}\leq \epsilon M_{*}(t)+ K_{\epsilon}t.
       \end{equation}
  \end{enumerate}
\end{lemma}

\begin{proof}
The proof is essentially contained in [\cite{Azroul}, Lemma 3.1].
\end{proof}

\begin{lemma}\label{ceb2}
  Let $\Omega$ be an open subset of $\mathbb{R}^{d}$ and $s\in(0,1)$. If $u\in W^{s,M}(\Omega)$ and $f$ is Lipschitz continuous on $\mathbb{R}$ such that $f(0)=0$, then $f\circ u$ belongs to $W^{s,M}(\Omega)$.
\end{lemma}

\begin{proof}
  Let $u\in  W^{s,M}(\Omega)$, then there exists $\theta>0$ such that $$\int_{\Omega} M(\theta u(x))dx<\infty\ \ \text{and}\ \ \int_{\Omega}\int_{\Omega} M\bigg{(}\frac{|u(x)-u(y)|}{|x-y|^{s}}\bigg{)}\frac{dxdy}{|x-y|^{d}}<\infty.$$
Let $K > 0$ denotes the Lipschitz constant of $f$, since $f(0)=0$, then
$$\int_{\Omega}M\bigg{(}\frac{\theta}{K} (f\circ u(x))\bigg{)}dx=\int_{\Omega}M\bigg{(}\frac{\theta}{K} (f\circ u(x)-f(0))\bigg{)}dx\leq
\int_{\Omega}M(\theta u(x))dx<\infty.$$ Therefore $f\circ u\in L^{M}(\Omega)$.\\

Now, let $\lambda=\frac{1}{K}$, we have
\begin{align*}
  \int_{\Omega}\int_{\Omega} M\bigg{(}\frac{\lambda|f(u(x))-f(u(y))|}{|x-y|^{s}}\bigg{)}\frac{dxdy}{|x-y|^{d}} & \leq \int_{\Omega}\int_{\Omega} M\bigg{(}\frac{|u(x)-u(y)|}{|x-y|^{s}}\bigg{)}\frac{dxdy}{|x-y|^{d}}<\infty.
\end{align*}
this implies that $\lambda (f\circ u)\in W^{s,M}(\Omega)$, thus $\lambda^{-1}\lambda (f\circ u)=f\circ u\in W^{s,M}(\Omega)$. This ends the proof.
\end{proof}

\begin{lemma}\label{ceb3}
  Let $\Omega$ be a bounded subset of $\mathbb{R}^{d}$ with $C^{0,1}$-regularity and bounded boundary. Let $M$ be an $N$-function satisfying condition $(m_{1})$. Then, given $0<s'<s<1$, it holds that the embedding
  \begin{equation}\label{s'}
    W^{s,M}(\Omega)\hookrightarrow W^{s',1}(\Omega),
  \end{equation}
  is continuous.
\end{lemma}

\begin{proof}
We closely follow the method employed in [\cite{B}, Proposition 2.9]. The normalization condition $M(1) = 1$ is by no means restrictive. From Lemma \ref{lem2} it is inferred that,
\begin{equation}\label{aa}
  t\leq M(t),\ \text{for every}\ t\geq1.
\end{equation}
  Let $u\in W^{s,M}(\Omega)$ such that $[u]_{(s,M)}=1$ and define $$A:=\{(x,y)\in \Omega\times\Omega:\ |h_{u}^{s}(x,y)|\leq1\},$$
$$B:= (\Omega\times\Omega)\backslash A,$$ where $h_{u}^{s}(x,y)=\displaystyle\frac{u(x)-u(y)}{|x-y|^{s}}.$\\
We compute,
\begin{align*}
  \int_{\Omega}\int_{\Omega}|h_{u}^{s'}(x,y)|\frac{dxdy}{|x-y|^{d}} & =  \int_{\Omega}\int_{\Omega}|h_{u}^{s}(x,y)||(x-y)|^{s-s'}\frac{dxdy}{|x-y|^{d}}\\
                                                  &=\bigg{(}\int\int_{A}+\int\int_{B}\bigg{)}|h_{u}^{s}(x,y)|\frac{dxdy}{|x-y|^{d-(s-s')}}\\
                                                   &:=J_{1}+J_{2},
\end{align*}
Notice that $$J_{1}\leq \int_{\Omega}\int_{\Omega}\frac{dxdy}{|x-y|^{d-(s-s')}}\leq \meas(\Omega)\omega_{d}\frac{\delta^{(s-s')}}{s-s'},$$ where $\delta$ is the diameter of $\Omega$.
To estimate the second term, we invoke \eqref{aa} and we obtain $$J_{2}\leq \delta^{(s-s')}\int_{\Omega}\int_{\Omega}M(h_{u}^{s}(x,y))\frac{dxdy}{|x-y|^{d}}<\infty.$$
Since $[u]_{(s,M)}=1$, then $\displaystyle\int_{\Omega}\int_{\Omega}M(h_{u}^{s}(x,y))\frac{dxdy}{|x-y|^{d}}\leq1$ and hence
\begin{equation}\label{sss}
  [u]_{s',1}\leq \bigg{(}\frac{\meas(\Omega)\omega_{d}}{s-s'}+1\bigg{)}\delta^{(s-s')}.
\end{equation}
 For arbitrary $u\in W^{s,M}(\Omega)\setminus\{0\}$, let $v=\frac{u}{[u]_{(s,M)}}$, using inequality \eqref{sss} we get $$[v]_{s',1}\leq \bigg{(}\frac{\meas(\Omega)\omega_{d}}{s-s'}+1\bigg{)}\delta^{(s-s')}.$$
By homogeneity of the seminorm $[.]_{s',1}$, we obtain
\begin{equation}\label{ineee1}
  [u]_{s',1}\leq \bigg{(}\frac{\meas(\Omega)\omega_{d}}{s-s'}+1\bigg{)}\delta^{(s-s')}[u]_{(s,M)}.
\end{equation}
On the other hand, since $\Omega$ is bounded, there exist $C>0$ such that
\begin{equation}\label{ineee2}
  \|u\|_{L^{1}(\Omega)}\leq C\|u\|_{(M)}\ (\text{see}\ [\cite{Rao},\ \text{Corollary}\ 3]).
\end{equation}
Combining \eqref{ineee1} and \eqref{ineee2} we get the desired result.
\end{proof}

\begin{proof}[{\bf Proof of theorem \ref{ceb}}]
  Let $u\in W^{s,M}(\Omega)\setminus\{0\}$ and suppose for the moment that $u$ is bounded on $\Omega$. Then $\lambda\rightarrow\int_{\Omega}M_{*}(|u(x)|/\lambda)dx$ decreases continuously from infinity to zero as $\lambda$ increases from zero to infinity. So that
  \begin{equation}\label{k}
    \int_{\Omega}M_{*}\bigg{(}\frac{|u(x)|}{k}\bigg{)}dx=1\ \ \ \ \text{for some}\ k>0,.
  \end{equation}
 By the definition of the norm \eqref{19}, we see that  $k=\|u\|_{(M_{*})}$.\\
  Let $\omega(t)=[M_{*}(t)]^{\frac{d-s'}{d}}$ and set $f(x)=\omega\bigg{(}\displaystyle\frac{|u(x)|}{k}\bigg{)}$. According to Lemma \ref{ceb1}, $\omega$ is Lipschitz continuous, by lemma \ref{ceb2} and  Lemma \ref{ceb3}, $f\in W^{s,M}(\Omega)\cap W^{s',1}(\Omega)$. 
   The well known embedding theorem of the classical fractional Sobolev space $W^{s',1}(\Omega)$ (see [\cite{23}, Theorem 6.7]), gives $$W^{s',1}(\Omega)\hookrightarrow L^{\frac{d}{d-s'}}(\Omega).$$
  So there is a constant $C_{1}>0$ such that
  \begin{align}
     1&= \bigg{(}\int_{\Omega}M_{*}\bigg{(}\frac{|u(x)|}{k}\bigg{)}dx\bigg{)}^{\frac{d-s'}{d}}=\|f\|_{L^{\frac{d}{d-s^{'}}}}\leq C_{1}(\|f\|_{L^{1}(\Omega)}+[f]_{s',1})\nonumber\\
      &=C_{1}\bigg{(}\int_{\Omega}\omega\bigg{(}\frac{|u(x)|}{k}\bigg{)}dx+\int_{\Omega}\int_{\Omega}\frac{|f(x)-f(y)|}{|x-y|^{d+s'}}dxdy\bigg{)}\\
      &:=C_{1}I_{1}+C_{1}I_{2}.\nonumber
  \end{align}
On one hand, by \eqref{Kepsilon} and the H\"{o}lder inequality, for $\epsilon=\displaystyle\frac{1}{2C_{1}}$, we have
\begin{align}\label{k1M}
  C_{1}I_{1}&\leq \frac{1}{2}\int_{\Omega}M_{*}\bigg{(}\frac{|u(x)|}{k}\bigg{)}dx+\frac{C_{1}K_{\epsilon}}{k}\int_{\Omega}|u(x)|dx\nonumber\\
            &\leq \frac{1}{2}+\frac{C_{2}}{k}\|u\|_{(M)},
\end{align}
 \text{where} $C_{2}=2C_{1}K_{\epsilon}\|\chi_{\Omega}\|_{(\overline{M})}$.\\

On the other hand, since $\omega$ is Lipschitz continuous, there exists $K>0$ such that,
$$C_{1}I_{2}\leq \frac{C_{1}K}{k}\int_{\Omega}\int_{\Omega}\frac{|u(x)-u(y)|}{|x-y|^{N+s'}}dxdy=\frac{C_{1}K}{k}[u]_{s',1}.$$
By \eqref{ineee1}, since $s^{'}<s$, there exists $C_{3}>0$ such that  $$[u]_{s',1}\leq C_{3} [u]_{(s,M)},$$ so
\begin{equation}\label{k1Mf}
  C_{1}I_{2}\leq \frac{C_{4}}{k}[u]_{(s,M)},
\end{equation}
where $C_{4}=KC_{1}C_{3}$. Combining \eqref{k1M} and \eqref{k1Mf}, we obtain
$$1\leq \frac{1}{2}+\frac{C_{2}}{k}\|u\|_{(M)}+\frac{C_{4}}{k}[u]_{(s,M)},$$
this implies that, $$\frac{k}{2}\leq C_{2}\|u\|_{(M)}+C_{4}[u]_{(s,M)}.$$
from which it follows that \begin{equation}\label{injection}
                             \|u\|_{(M_{*})}\leq C_{5}\|u\|_{(s,M)},
                           \end{equation}
                           where $C_{5}=\max\{2C_{2},2C_{4}\}$.\\

To extend \eqref{injection} to arbitrary $u\in W^{s,M}(\Omega)$, let
$$f_{n}(y)=\begin{cases}
             y & \mbox{if }\ |y|\leq n \\
             n\ sgn(y)\ & \mbox{if}\ |y|>n
           \end{cases}\ \ \ \
           \text{and}\ \ \ \ u_{n}(x)=f_{n}\circ u(x).$$
 Clearly $f_{n}$ is $1$-Lipschitz continuous function. By Lemma \ref{ceb2}, $(u_{n})$ belongs to $W^{s,M}(\Omega)$.
 So in view of \eqref{injection}
 \begin{equation}\label{000}
   \|u_{n}\|_{(M_{*})}\leq C_{5}\|u_{n}\|_{(s,M)}.
 \end{equation}
 On the other hand, we have
\begin{equation}\label{0000}
  \|u_{n}\|_{(s,M)}\leq \|u\|_{(s,M)},
\end{equation}
indeed, since $|u_{n}(x)|\leq |u(x)|$, for all $x\in\Omega$, then
 \begin{align*}
          \int_{\Omega}\int_{\Omega}M\bigg{(}\frac{|u_{n}(x)-u_{n}(y)|}{[u]_{(s,M)}|x-y|^{s}}\bigg{)}\frac{dxdy}{|x-y|^{N}} & \leq \int_{\Omega}\int_{\Omega}M\bigg{(}\frac{|u(x)-u(y)|}{[u]_{(s,M)}|x-y|^{s}}\bigg{)}\frac{dxdy}{|x-y|^{N}}\leq1,
        \end{align*}
        and $$\int_{\Omega}M\bigg{(}\frac{|u_{n}(x)|}{\|u\|_{(M)}}\bigg{)}dx\leq \int_{\Omega}M\bigg{(}\frac{|u(x)|}{\|u\|_{(M)}}\bigg{)}dx\leq1,$$
        then $$[u_{n}]_{(s,M)}\leq [u]_{(s,M)}\ \ \text{and}\ \ \|u_{n}\|_{(M)}\leq\|u\|_{(M)},$$
        thus \eqref{0000} is deduced. Combining \eqref{000} and \eqref{0000}, we get
 \begin{equation}\label{ineq}
   \|u_{n}\|_{(M_{*})}\leq C_{5}\|u_{n}\|_{(s,M)}\leq C_{5}\|u\|_{(s,M)}.
 \end{equation} Let $k_{n}=\|u_{n}\|_{(M_{*})}$, the sequence $(k_{n})$ is non-decreasing and converges in view of \eqref{ineq}. Put $k'=\lim_{n\rightarrow+\infty}k_{n}$,
 by Fatou's Lemma we get
 $$\int_{\Omega}M_{*}\bigg{(}\frac{u(x)}{k'}\bigg{)}dx\leq \lim_{n\rightarrow+\infty}\int_{\Omega}
 M_{*}\bigg{(}\frac{u_{n}(x)}{k_{n}}\bigg{)}dx\leq 1,$$
 whence $u\in L^{M_{*}}(\Omega)$ and $$\|u\|_{(M_{*})}\leq k'=\lim_{n\rightarrow+\infty}\|u_{n}\|_{(M_{*})}\leq C_{4}\|u\|_{(s,M)}.$$
Thus the first assertion of the theorem is proved. Now, let's turn to the compactness embedding.\\

Let $S$ be a bounded subset of $W^{s,M}(\Omega)$. According to the embedding \eqref{7}, $S$ is also a bounded subset of $L^{M_{*}}(\Omega)$. On the other hand, by a classical compact embedding theorem of $W^{s',1}(\Omega)$ (see [\cite{23}, Corollary 7.2]) and Lemma \ref{ceb2}, we have $$W^{s,M}(\Omega)\hookrightarrow W^{s',1}(\Omega)\hookrightarrow L^{1}(\Omega).$$ Then, $S$ is precompact in $L^{1}(\Omega)$. Hence, by [ \cite{Adams}, Theorem 8.25], $S$ is precompact in $L^{B}(\Omega)$  whenever $B\prec\prec M_{*}$. The theorem is proved completely.
\end{proof}

\subsection{Equivalent norm in $W^{s,M}(\Omega)$}

Let $\Omega$ be an open subset of $\mathbb{R}^{d}$ and $u\in W^{s,M}(\Omega)$. Let $$\rho(u)=\int_{\Omega}M(u(x))dx, \ \ \overline{\rho}(u)=\int_{\Omega}\int_{\Omega}M\bigg{(}\frac{u(x)-u(y)}{|x-y|^{s}}\bigg{)}\frac{dxdy}{|x-y|^{N}},\ \  \ \tilde{\rho}(u)=\rho(u)+\overline{\rho}(u)$$
and
$$|u|_{(s,M,\Omega)}=\inf\bigg{\{}\lambda>0\ :\ \tilde{\rho}\bigg{(}\frac{u}{\lambda}\bigg{)}\leq1\bigg{\}}.$$

\begin{remark}\label{norm}
  We can notice using the Fatou's lemma that $$\tilde{\rho}\bigg{(}\frac{u}{|u|_{(s,M,\Omega)}}\bigg{)}\leq1,\ \text{for all}\ u\in W^{s,M}(\Omega).$$
\end{remark}

\begin{lemma}\label{3.5}
 $|.|_{(s,M,\Omega)}$ is an equivalent norm to $\|.\|_{(s,M,\Omega)}$ with the relation
 \begin{equation}\label{equivalence}
   \frac{1}{2}\|u\|_{(s,M,\Omega)}\leq |u|_{(s,M,\Omega)}\leq 2\|u\|_{(s,M,\Omega)},\ \text{for all}\ u\in W^{s,M}(\Omega).
 \end{equation}
\end{lemma}

\begin{proof}
We begin by proving \eqref{equivalence}. Evidently, we have
  $$\rho\bigg{(}\frac{u}{|u|_{(s,M,\Omega)}}\bigg{)}\leq\tilde{\rho}\bigg{(}\frac{u(x)}{|u|_{(s,M,\Omega)}}\bigg{)}\leq1\ \ \ \text{and}\ \ \
  \overline{\rho}\bigg{(}\frac{u}{|u|_{(s,M,\Omega)}}\bigg{)}\leq\tilde{\rho}\bigg{(}\frac{u}{|u|_{(s,M,\Omega)}}\bigg{)} \leq1,$$
  then $$\|u\|_{(M,\Omega)}\leq |u|_{(s,M,\Omega)}\ \ \text{and}\ \ [u]_{(s,M,\Omega)}\leq |u|_{(s,M,\Omega)},$$
  therefore
  $$\|u\|_{(s,M,\Omega)}\leq 2|u|_{(s,M,\Omega)}.$$
  For the second inequality of \eqref{equivalence}, we have
  \begin{align*}
    \tilde{\rho}\bigg{(}\frac{u}{2\|u\|_{(s,M,\Omega)}}\bigg{)} & \leq \frac{1}{2}\rho\bigg{(}\frac{u}{\|u\|_{(s,M,\Omega)}}\bigg{)}+
                                                               \frac{1}{2}\overline{\rho}\bigg{(}\frac{u}{\|u\|_{(s,M,\Omega)}}\bigg{)}\\
                                                               &\leq\frac{1}{2}\rho\bigg{(}\frac{u}{\|u\|_{(M,\Omega)}}\bigg{)}+
                                                               \frac{1}{2}\overline{\rho}\bigg{(}\frac{u}{[u]_{(s,M,\Omega)}}\bigg{)}\\
                                                               &\leq \frac{1}{2}+\frac{1}{2}=1,
  \end{align*}
  then  $$|u|_{(s,M,\Omega)}\leq2\|u\|_{(s,M,\Omega)}.$$ This ends the proof of \eqref{equivalence}.\\

   Let now prove that $|.|_{(s,M,\Omega)}$ is a norm in $W^{s,M}(\Omega)$.\\

 \noindent $(\mathbf{i})$ It is clear that, if $|u|_{(s,M,\Omega)}=0$ then $u=0,\ a.e$.\\
  $(\mathbf{ii})$ For $\alpha\in \mathbb{K}$, we have
  \begin{align*}
    |\alpha u|_{(s,M,\Omega)} & =\inf\bigg{\{}\lambda>0\ :\ \tilde{\rho}\bigg{(}\frac{\alpha u}{\lambda}\bigg{)}\leq1\bigg{\}}=
                               |\alpha| \inf\bigg{\{}\lambda>0\ :\ \tilde{\rho}\bigg{(}\frac{u}{\lambda}\bigg{)}\leq1\bigg{\}}\\
                               &=|\alpha| |u|_{(s,M,\Omega)}.
  \end{align*}

\noindent$(\mathbf{iii})$ Finally for the triangle inequality, let $u,v\in W^{s,M}(\Omega)$, we compute
  \begin{align*}
    \tilde{\rho}\bigg{(}\frac{u+v}{|u|_{(s,M,\Omega)}+|v|_{(s,M,\Omega)}}\bigg{)} & =\tilde{\rho}\bigg{(}\frac{|u|_{(s,M,\Omega)}}{|u|_{(s,M,\Omega)}+|v|_{(s,M,\Omega)}}\frac{u}{|u|_{(s,M,\Omega)}}
    +\frac{|v|_{(s,M,\Omega)}}{|u|_{(s,M,\Omega)}+|v|_{(s,M,\Omega)}}\frac{v}{|v|_{(s,M,\Omega)}}\bigg{)} \\ &\leq \frac{|u|_{(s,M,\Omega)}}{|u|_{(s,M,\Omega)}+|v|_{(s,M,\Omega)}} \tilde{\rho}\bigg{(}\frac{u}{|u|_{(s,M,\Omega)}}\bigg{)}+\frac{|v|_{(s,M,\Omega)}}{|u|_{(s,M,\Omega)}+|v|_{(s,M,\Omega)}}
    \tilde{\rho}\bigg{(}\frac{v}{|v|_{(s,M,\Omega)}}\bigg{)}\\&\leq1.
  \end{align*}
  Thus $$|u+v|_{(s,M,\Omega)}\leq |u|_{(s,M,\Omega)}+|v|_{(s,M,\Omega)}.$$ The proof of Lemma \ref{3.5} is completed.
\end{proof}

\begin{lemma}\label{rhotilde}
The following properties hold true:
  \begin{enumerate}
    \item [(i)] If $|u|_{(s,M,\Omega)}>1$, then $|u|_{(s,M,\Omega)}^{m_{0}}\leq \tilde{\rho}(u)\leq |u|_{(s,M,\Omega)}^{m^{0}}.$
    \item [(ii)] If $|u|_{(s,M,\Omega)}<1$, then  $|u|_{(s,M,\Omega)}^{m^{0}}\leq \tilde{\rho}(u)\leq |u|_{(s,M,\Omega)}^{m_{0}}.$
  \end{enumerate}
\end{lemma}

\begin{proof}
 ${\bf (1)}$  Assume that $|u|_{(s,M,\Omega)}>1$, then by Lemma \ref{lem2} and Remark \ref{norm},
  \begin{align*}
    \tilde{\rho}\bigg{(}u\bigg{)} & =
    \int_{\Omega}M\bigg{(}|u|_{(s,M,\Omega)}\frac{u}{|u|_{(s,M,\Omega)}}\bigg{)}dx+
    \int_{\Omega}\int_{\Omega}M\bigg{(}|u|_{(s,M,\Omega)}\frac{h_{u}(x,y)}{|u|_{(s,M,\Omega)}}\bigg{)}\frac{dxdy}{|x-y|^{N}}\\
    &\leq |u|_{(s,M,\Omega)}^{m^{0}}\bigg{[}\int_{\Omega}M\bigg{(}\frac{u}{|u|_{(s,M,\Omega)}}\bigg{)}dx+
    \int_{\Omega}\int_{\Omega}M\bigg{(}\frac{h_{u}(x,y)}{|u|_{(s,M,\Omega)}}\bigg{)}\frac{dxdy}{|x-y|^{N}}\bigg{]}\\
    &= |u|_{(s,M,\Omega)}^{m^{0}}\tilde{\rho}\bigg{(}\frac{u}{|u|_{(s,M,\Omega)}}\bigg{)}\leq|u|_{(s,M,\Omega)}^{m^{0}},
  \end{align*}
  where $h_{u}(x,y)=\displaystyle\frac{u(x)-u(y)}{|x-y|^{s}}.$\\

  Let $1<\sigma<|u|_{(s,M,\Omega)}$, by definition of the norm $|.|_{(s,M,\Omega)}$,
  $$\tilde{\rho}\bigg{(}\frac{u}{\sigma}\bigg{)}>1, $$
  then, according to Lemma \ref{lem2}, we infer
  \begin{align*}
    \tilde{\rho}(u)&\geq \sigma^{m_{0}}\bigg{[}\int_{\Omega}M\bigg{(}\frac{u}{\sigma}\bigg{)}dx+
    \int_{\Omega}\int_{\Omega}M\bigg{(}\frac{h_{u}(x,y)}{\sigma}\bigg{)}dxdy\bigg{]}\\
    &= \sigma^{m_{0}}\tilde{\rho}\bigg{(}\frac{u}{\sigma}\bigg{)}\geq\sigma^{m_{0}},
  \end{align*}
  letting $\sigma\nearrow|u|_{(s,M,\Omega)}$ in the above inequality, we obtain $(i)$.\\

  ${\bf (2)}$ Assume that $|u|_{(s,M,\Omega)}<1$. Using Lemma \ref{lem2}, we get
  \begin{align*}
     \tilde{\rho}(u) & =\tilde{\rho}\bigg{(}|u|_{(s,M,\Omega)}\frac{u}{|u|_{(s,M,\Omega)}}\bigg{)}
      \leq |u|_{(s,M,\Omega)}^{m_{0}}\tilde{\rho}\bigg{(}\frac{u}{|u|_{(s,M,\Omega)}}\bigg{)}
      \leq  |u|_{(s,M,\Omega)}^{m_{0}}.
  \end{align*}
 On the other hand, as above in $(1)$,
  let $0<\sigma<|u|_{(s,M,\Omega)}$, by Lemma \ref{lem2},
  $$\tilde{\rho}(u)\geq\sigma^{m^{0}} \tilde{\rho}\bigg{(}\frac{u}{\sigma}\bigg{)}\geq \sigma^{m^{0}}.$$
Letting $\sigma\nearrow|u|_{(s,M,\Omega)}$ in the last inequality, we obtain $(ii)$.
\end{proof}

\begin{proof}[{\bf Proof of Theorem \ref{allR}}]
  Let $u\in W^{s,M}(\mathbb{R}^{d})$ such that $|u|_{(s,M,\mathbb{R}^{d})}=1$. By Lemma \eqref{rhotilde}, $\tilde{\rho}(u)=1$. Let $B_{i}=\{x\in\mathbb{R}^{d}/\ i\leq|x|<i+1\}$, $i\in\mathbb{N}$, such that $\mathbb{R}^{d}= \bigcup_{i\in\mathbb{N}}B_{i}$ and $B_{i}\cap B_{j}\neq\emptyset$ if $i\neq j$.
 Then, for all $i\in\mathbb{N}$, we have
  \begin{align}\label{tilder}
 \int_{B_{i}}M(u(x))dx+ \int_{B_{i}} \int_{B_{i}}M\bigg{(}\frac{u(x)-u(y)}{|x-y|^{s}M^{-1}(|x-y|^{N})}\bigg{)}dxdy & \leq\tilde{\rho}(u)=1.
   \end{align}

 In what follows we show that $$C=\sup\bigg{\{}\int_{\mathbb{R}^{d}}M_{*}(w(x))dx:\ w\in W^{s,M}(\mathbb{R}^{d}),|w|_{(s,M,\mathbb{R}^{d})}=1\bigg{\}}<\infty.$$
Indeed, in view of \eqref{7} and \eqref{equivalence}, there exists $C_{0}>$ such that  $$\|u\|_{(M_{*},B_{i})}\leq C_{0}\|u\|_{(s,M,B_{i})}\leq2C_{0}|u|_{(s,M,B_{i})}\leq 2C_{0}|u|_{(s,M,{\mathbb{R}^{d}})}=2C_{0},\ \text{for all}\ i\in\mathbb{N}.$$
Let $i\in\mathbb{N}$, we distingue two cases:\\

  {\bf Cas 1:} If $1\leq\|u\|_{(M_{*},B_{i})}\leq2C_{0}$, then, by Lemma \ref{lem2}, \eqref{7}, \eqref{equivalence}, Lemma \ref{rhotilde} and  \eqref{tilder}, we have
  \begin{align}\label{en1}
 \int_{B_{i}}M_{*}(u(x))dx& \leq \|u\|_{(M_{*},B_{i})}^{(m^{0})^{*}}\leq C_{0}^{(m^{0})^{*}}\|u\|_{(s,M,B_{i})}^{(m^{0})^{*}}\ \ \big{(}(m^{0})^{*}\ \text{fixed by Lemma \ref{M***}}\big{)}\nonumber\\
       &\leq (2C_{0})^{(m^{0})^{*}}|u|_{(s,M,B_{i})}^{(m^{0})^{*}}\nonumber\\
       &\leq (2C_{0})^{(m^{0})^{*}}\bigg{(} \int_{B_{i}}M(u(x))dx+ \int_{B_{i}} \int_{B_{i}}M(h_{u}(x,y))dxdy \bigg{)}^{\frac{(m^{0})^{*}}{m^{0}}}\nonumber\\
       &\leq (2C_{0})^{(m^{0})^{*}}\bigg{(} \int_{B_{i}}M(u(x))dx+ \int_{B_{i}} \int_{B_{i}}M(h_{u}(x,y))dxdy \bigg{)}.
  \end{align}

 {\bf Cas 2:} If $\|u\|_{(M_{*},B_{i})}<1$, then, also by Lemma \ref{lem2}, \eqref{7}, \eqref{equivalence}, Lemma \ref{rhotilde} and \eqref{tilder}, we have

 \begin{align}\label{en2}
 \int_{B_{i}}M_{*}(u(x))dx& \leq \|u\|_{(M_{*},B_{i})}^{m_{0}^{*}}\leq C_{0}^{m_{0}^{*}}\|u\|_{(s,M,B_{i})}^{m_{0}^{*}}\ \ \big{(}m_{0}^{*}\ \text{fixed by Lemma \ref{M***}}\big{)}\nonumber\\
       &\leq (2C_{0})^{m_{0}^{*}}|u|_{(s,M,B_{i})}^{m_{0}^{*}}\nonumber\\
       &\leq (2C_{0})^{m_{0}^{*}}\bigg{(} \int_{B_{i}}M(u(x))dx+ \int_{B_{i}} \int_{B_{i}}M(h_{u}(x,y))dxdy \bigg{)}^{\frac{m_{0}^{*}}{m^{0}}}\nonumber\\
       &\leq (2C_{0})^{m_{0}^{*}}\bigg{(} \int_{B_{i}}M(u(x))dx+ \int_{B_{i}} \int_{B_{i}}M(h_{u}(x,y))dxdy \bigg{)}.
  \end{align}
  Combining \eqref{en1} and \eqref{en2}, we obtain \begin{align*}
         \int_{\mathbb{R}^{d}}M_{*}(u(x))dx & =\sum_{i\in\mathbb{N}}\int_{B_{i}}M_{*}(u(x))dx\\
                                            &\leq[(2C_{0})^{(m^{0})^{*}}+(2C_{0})^{m_{0}^{*}}]\sum_{i\in\mathbb{N}} \bigg{(} \int_{B_{i}}M(u(x))dx+ \int_{B_{i}} \int_{B_{i}}M(h_{u}(x,y))dxdy \bigg{)}\\
                                            &= [(2C_{0})^{(m^{0})^{*}}+(2C_{0})^{m_{0}^{*}}]\tilde{\rho}(u)=(2C_{0})^{(m^{0})^{*}}+(2C_{0})^{m_{0}^{*}}.
       \end{align*}
      Hence $C<(2C_{0})^{(m^{0})^{*}}+(2C_{0})^{m_{0}^{*}}.$\\

        Now, let $u\in W^{s,M}(\mathbb{R}^{d})\setminus\{0\}$ and $v=\displaystyle\frac{u}{|u|_{(s,M,\mathbb{R}^{d})}}$. Then, by using Proposition \ref{KML}, we infer $$\|v\|_{(M_{*})}\leq \int_{\mathbb{R}^{d}}M_{*}(v(x))dx+1\leq C+1,$$
  from where it follows that $$\|u\|_{(M_{*})}\leq (C+1)|u|_{(s,M,\mathbb{R}^{d})},$$
  and for all $N$-function $B\prec\prec M_{*}$, we have $$W^{s,M}(\mathbb{R}^{d})\hookrightarrow L^{M_{*}}(\mathbb{R}^{d})\hookrightarrow L^{B}(\mathbb{R}^{d}).$$
  The proof of Theorem \ref{allR} is completed.
\end{proof}

\section{Variational setting of problem \eqref{eq1} and some useful tools}

In this section, we will first introduce the variational setting for problem \eqref{eq1}. In view of the presence of potential $V$, our working space is $$E=\bigg{\{}u\in W^{s,M}(\mathbb{R}^{d});\ \int_{\mathbb{R}^{d}}V(x)M(u)dx<\infty\bigg{\}},$$
equipped with the following norm $$\|u\|= [u]_{(s,M)}+\|u\|_{(V,M)},$$ where $$\|u\|_{(V,M)}=\inf\bigg{\{}\lambda>0;\ \int_{\mathbb{R}^{d}}V(x)M\bigg{(}\frac{u(x)}{\lambda}\bigg{)}dx\leq1\bigg{\}}.$$
We define the functional $G: E\rightarrow \mathbb{R}$ by
\begin{equation}\label{F}
G(u)=\int_{\mathbb{R}^{d}}\int_{\mathbb{R}^{d}}M(h_{u}(x,y))d\mu,
\end{equation}
where $h_{u}(x,y)=\displaystyle\frac{u(x)-u(y)}{|x-y|^{s}}$ and $d\mu=\displaystyle\frac{dxdy}{|x-y|^{N}}$.\\

For any $(x,t)\in \mathbb{R}^{d}\times\mathbb{R}$, denote by
\begin{equation}\label{F}
  F(x,t)=\int_{0}^{t}f(x,s)ds=\xi(x)|t|^{p}.
\end{equation}

We consider the following family of functionals on $E$
$$I_{\lambda}(u)=A(u)-\lambda B(u),\ \lambda\in[1,2],$$
where $$A(u)=G(u)+\Psi(u),$$

 $$\Psi(u)=\int_{\mathbb{R}^{d}}V(x)M(u)dx,\ \ \ \ B(u)=\int_{\mathbb{R}^{d}}F(x,u)dx.$$

\begin{lemma}\label{deff}
 The functional $I_{\lambda}$ is well defined on $E$, moreover $I_{\lambda}\in C^{1}(E,\mathbb{R})$ and for all $v\in E$,\begin{align}\label{I'}
 \langle I_{\lambda}^{'}(u),v\rangle&=\langle A^{'}(u),v\rangle-\lambda\langle B^{'}(u),v\rangle\nonumber\\
                                     &=\langle(-\triangle)^{s}_{m}u,v\rangle+\int_{\mathbb{R}^{d}}V(x)m(u)vdx-
                                     \lambda\int_{\mathbb{R}^{d}}f(x,u)vdx.
\end{align}
\end{lemma}

\begin{proof}
   The proof follows from [\cite{Zhang}, Proposition $2.2$] and [\cite{Salort}, Proposition $3.3$].
\end{proof}

Now we give the definition of weak solution for the problem \eqref{eq1}.

\begin{definition}
  We say that $u\in E$ is a weak solution to $\eqref{eq1}$ if $u$ is critical point of $I_{1}$, which means that $u$ satisfies $$\langle I_{1}^{'}(u),v\rangle=\langle A^{'}(u),v\rangle-\langle B^{'}(u),v\rangle=0$$ for all $v\in E$.
\end{definition}

\begin{lemma}\label{lem3} Assume that $(m_{1})$ and $(V_{1})$ are satisfied. Then,
  the following properties hold true:\\

    \noindent$(i)$ $\xi_{0}([u]_{(s,M)})\leq G(u)\leq\xi_{1}([u]_{(s,M)})\ \text{for all}\ u\in E,$\\

    \noindent$(ii)$ $\xi_{0}(\|u\|_{(V,M)})\leq \displaystyle\int_{\mathbb{R}^{d}}V(x)M(u)dx\leq \xi_{1}(\|u\|_{(V,M)})\ \text{for all}\ u\in E,$
\end{lemma}

\begin{proof}
 The proof of the first assertion is given by [\cite{Sabri}, Lemma 3.4]. For the second assertion, let $u\in E$, on one hand, choosing $\beta=\|u\|_{(V,M)}$ in Lemma \ref{lem2}, we obtain
   $$
    M( u)\leq \xi_{1}(\|u\|_{(V,M)})M\bigg{(}\frac{u}{\|u\|_{(V,M)}}\bigg{)},$$
   then $$
   V(x)M( u)\leq \xi_{1}(\|u\|_{(V,M)})V(x)M\bigg{(}\frac{u}{\|u\|_{(V,M)}}\bigg{)},\ \text{for}\ x\in\mathbb{R}^{d}.$$
   From the definition of the norm \eqref{19}, we deduce that

   $$\int_{\mathbb{R}^{d}}V(x)M( u)dx\leq \xi_{1}(\|u\|_{(V,M)})\int_{\mathbb{R}^{d}}V(x)M\bigg{(}\frac{u}{\|u\|_{(V,M)}}\bigg{)}dx\leq \xi_{1}(\|u\|_{(V,M)}).$$
   On the other hand, let $\epsilon>0$, $\beta=\|u\|_{(V,M)}-\epsilon$ in Lemma \ref{lem2}, we get
   $$\xi_{0}(\|u\|_{(V,M)}-\epsilon)V(x)M\bigg{(}\frac{u}{\|u\|_{(V,M)}-\epsilon}\bigg{)}\leq V(x)M( u),$$
   then, $$\int_{\mathbb{R}^{d}}V(x)M( u)dx\geq \xi_{0}(\|u\|_{(V,M)}-\epsilon)\int_{\mathbb{R}^{d}}V(x)M\bigg{(}\frac{u}{\|u\|_{(V,M)}-\epsilon}\bigg{)}dx\geq\xi_{0}(\|u\|_{(V,M)}-\epsilon).$$
   Letting $\epsilon\rightarrow0$ in the above inequality, we obtain $$\xi_{0}(\|u\|_{(V,M)})\leq\int_{\mathbb{R}^{d}}V(x)M( u)dx.$$
   Thus the assertion $(ii)$ and the proof of Lemma \ref{lem3} is complete.
\end{proof}

\begin{lemma}\label{stronger}
   Let $\varphi,\psi:\mathbb{R}\rightarrow\mathbb{R}$ be increasing homeomorphisms such that their associated $N$-functions $\Phi,\Psi$ satisfy 
  \begin{equation}\label{m}
    1<\varphi_{0}:=\inf_{t>0}\frac{t\varphi(t)}{\Phi(t)}\leq
    \frac{t\varphi(t)}{\Phi(t)}\leq \varphi^{0}:=\sup_{t>0}\frac{t\varphi(t)}{\Phi(t)}<m_{0},\ \ \ (m_{0}\ \text{fixed by $(m_{1})$}),
  \end{equation}
 and
 \begin{equation}\label{mm}
   1<\frac{t \psi(t)}{\Psi(t)}\leq \psi^{0}:=\sup_{t>0}\frac{t \psi(t)}{\Psi(t)}<\frac{m_{0}}{\varphi^{0}}.
 \end{equation}

  Then $\Psi$ satisfies the $\triangle_{2}$-condition and $$\Psi\circ\Phi\prec\prec M.$$
\end{lemma}

\begin{proof}
  Using Lemma \ref{lem2}, we have, for all $\lambda>0$ and $t>1$,
  $$\Phi(\lambda t)\leq \Phi(\lambda)t^{\varphi^{0}}\ \ \text{and}\ \ M(1)t^{m_{0}}\leq M(t),$$ then
  $$\frac{\Phi(\lambda t)}{M(t)}\leq \frac{\Phi(\lambda)t^{\varphi^{0}}}{M(1)t^{m_{0}}}.$$
  Again with Lemma \ref{lem2}, we get $$\Psi(\Phi(\lambda t))\leq \Psi(\Phi(\lambda)t^{\varphi^{0}})\leq \Psi(\Phi(\lambda ))t^{\varphi^{0}\psi^{0}}.$$
   Since $\varphi^{0}\psi^{0}<m_{0}$, we conclude $$\frac{\Psi(\Phi(\lambda t))}{M(t)}\leq \frac{\Psi(\Phi(\lambda ))t^{\varphi^{0}\psi^{0}}}{M(1)t^{m_{0}}}\rightarrow0,\ \text{as}\ t\rightarrow+\infty.$$ Thus $\Psi\circ\Phi\prec\prec M$ and the proof is completed.
\end{proof}

Now we state our embedding compactness result.
\begin{lemma}\label{lem1}
 Let $\Phi$ be an $N$-function satisfying \eqref{m} such that
 \begin{equation}\label{Mphi}
\Phi\prec\prec M.
 \end{equation}
 Under the assumption $(m_{1})$, $(V_{1})$ and $(V_{2})$, the embedding from $E$ into $L^{\Phi}(\mathbb{R}^{d})$ is compact.
\end{lemma}
\begin{proof}
 Let $\Phi$ be an $N$-function satisfying \eqref{m} such that $\Phi\prec\prec M$ and $(v_{n})$ be a bounded sequence in $E$, since $E$ is reflexive, up to subsequence, $v_{n}\rightharpoonup v\ \text{in}\ E.$
  Let $u_{n}=v_{n}-v$, $u_{n}\rightharpoonup 0\ \text{in}\ E$. We have to show that $u_{n}\rightarrow0$ in $L^{\Phi}(\mathbb{R}^{d})$, by Proposition \ref{cv} this means that
  \begin{equation}\label{sab}
   \int_{\mathbb{R}^{d}}\Phi(u_{n})dx\rightarrow0,\ \text{as}\ n\rightarrow+\infty.
  \end{equation}
  According to Vitali's theorem it suffices to show that, the sequence $(\Phi(u_{n}))$ is equi-integrable, which means:

  $$\mathbf{(a)}\ \ \ \ \ \ \ \ \ \ \begin{cases}
               \forall\epsilon>0,\ \exists\ \omega\subset\mathbb{R}^{d}\ \text{measurable with}\ \meas(\omega)<\infty\\
               \text{such that}\ \int_{\mathbb{R}^{d}\setminus\omega}\Phi(u_{n}(x))dx<\epsilon,\ \forall n\in\mathbb{N},
             \end{cases}$$

  $$\mathbf{(b)}\ \ \ \ \ \ \ \ \ \ \begin{cases}
               \forall\epsilon>0,\ \exists\ \delta_{\epsilon}>0\ \text{such that}\ \int_{E}\Phi(u_{n}(x))dx<\epsilon,\\
               \forall n\in\mathbb{N},\ \forall\ E\subset\mathbb{R}^{d},\ E\ \text{measurable and}\ \meas(E)<\delta_{\epsilon}.
             \end{cases}$$

  We do the proof in two steps. We start by checking $\mathbf{(a)}$. Let $L>0$ and $A_{L}=\{x\in \mathbb{R}^{d}:\ V(x)\leq L\}.$\\
By \eqref{Mphi} and [\cite{Kufner}, Proposition $4.17.6$,], there exists $C>0$ such that
\begin{equation}\label{ttt}
  \|u_{n}\|_{(\Phi,A_{L}^{c})}\leq C\|u_{n}\|_{(M,A_{L}^{c})},\ \text{for all}\ n\in \mathbb{N}.
\end{equation}

Since $\Phi$ satisfying \eqref{m}, then by Lemma \ref{lem2} and \eqref{ttt}, we infer
  \begin{align}\label{ARL}
    \int_{A_{L}^{c}}\Phi(u_{n})dx &\leq \xi_{1}(\|u_{n}\|_{(\Phi,A_{L}^{c})})\leq\|u_{n}\|_{(\Phi,A_{L}^{c})}^{\varphi_{0}}+\|u_{n}\|_{(\Phi,A_{L}^{c})}^{\varphi^{0}}\nonumber\\
                                  &\leq (C^{\varphi_{0}}+C^{\varphi^{0}}) \big{(}\|u_{n}\|_{(M,A_{L}^{c})}^{\varphi_{0}}+\|u_{n}\|_{(M,A_{L}^{c})}^{\varphi^{0}}\big{)}.\nonumber\\
   \end{align}
 Applying again Lemma \ref{lem2}, we obtain
 \begin{align}\label{arl}
   \|u_{n}\|_{(M,A_{L}^{c})}^{\varphi_{0}}& \leq2^{\varphi_{0}-1}\bigg{[}\bigg{(}\int_{A_{L}^{c}}M(u_{n})dx\bigg{)}^{\frac{\varphi_{0}}{m^{0}}}+
                                  \bigg{(}\int_{A_{L}^{c}}M(u_{n})dx\bigg{)}^{\frac{\varphi_{0}}{m_{0}}}\bigg{]}.
 \end{align}
 Combining \eqref{ARL}, \eqref{arl} and Lemma \ref{lem3}, we get
   \begin{align*}
        \int_{A_{L}^{c}}\Phi(u_{n})dx&\leq (2^{\varphi_{0}-1}+ 2^{\varphi^{0}-1})(C^{\varphi_{0}}+C^{\varphi^{0}}) \bigg{[}\bigg{(}\int_{A_{L}^{c}}M(u_{n})dx\bigg{)}^{\frac{\varphi_{0}}{m^{0}}}+
                                  \bigg{(}\int_{A_{L}^{c}}M(u_{n})dx\bigg{)}^{\frac{\varphi_{0}}{m_{0}}}\nonumber\\
                                  &+\bigg{(}\int_{A_{L}^{c}}M(u_{n})dx\bigg{)}^{\frac{\varphi^{0}}{m^{0}}}+
                                  \bigg{(}\int_{A_{L}^{c}}M(u_{n})dx\bigg{)}^{\frac{\varphi^{0}}{m_{0}}}\bigg{]}\nonumber\\
                                  &\leq (2^{\varphi_{0}-1}+ 2^{\varphi^{0}-1})(C^{\varphi_{0}}+C^{\varphi^{0}}) \bigg{[}L^{-\frac{\varphi_{0}}{m^{0}}}\bigg{(}\int_{A_{L}^{c}}V(x)M(u_{n})dx\bigg{)}^{\frac{\varphi_{0}}{m^{0}}}+
                                  L^{-\frac{\varphi_{0}}{m_{0}}}\bigg{(}\int_{A_{L}^{c}}V(x)M(u_{n})dx\bigg{)}^{\frac{\varphi_{0}}{m_{0}}}\nonumber\\
                                  &+L^{-\frac{\varphi^{0}}{m^{0}}}\bigg{(}\int_{A_{L}^{c}}V(x)M(u_{n})dx\bigg{)}^{\frac{\varphi^{0}}{m^{0}}}+
                                  L^{-\frac{\varphi^{0}}{m_{0}}}\bigg{(}\int_{A_{L}^{c}}V(x)M(u_{n})dx\bigg{)}^{\frac{\varphi^{0}}{m_{0}}}\bigg{]}\nonumber\\
                                  &\leq (2^{\varphi_{0}-1}+ 2^{\varphi^{0}-1})(C^{\varphi_{0}}+C^{\varphi^{0}})\bigg{(}L^{-\frac{\varphi_{0}}{m^{0}}}
                                  \xi_{1}(\|u_{n}\|_{(V,M)})^{\frac{\varphi_{0}}{m^{0}}}
                                  +L^{-\frac{\varphi_{0}}{m_{0}}}\xi_{1}(\|u_{n}\|_{(V,M)})^{\frac{\varphi_{0}}{m_{0}}}\nonumber\\
                                  &+L^{-\frac{\varphi^{0}}{m^{0}}}\xi_{1}(\|u_{n}\|_{(V,M)})^{\frac{\varphi^{0}}{m^{0}}}
                                  +L^{-\frac{\varphi^{0}}{m_{0}}}\xi_{1}(\|u_{n}\|_{(V,M)})^{\frac{\varphi^{0}}{m_{0}}}\bigg{)}\nonumber\\
  \end{align*}
  this can be made arbitrarily small by choosing $L$ large enough. Thus $(a)$ is verified.\\

   Now we check $\mathbf{(b)}$. Let $B\subset\mathbb{R}^{d}$ measurable subset of $\mathbb{R}^{d}$. Let $\Psi$ be the $N$-function satisfying \eqref{mm} and $\overline{\Psi}$ be the conjugate of $\Psi$. By Lemma \ref{lem2}, $\Psi\circ\Phi\prec M$. In the light of [\cite{Kufner}, Theorem 4.17.6] there exists $K^{'}>0$ such that \begin{equation}\label{326}
                                                    \|u\|_{(\Psi\circ\Phi)}\leq K^{'}\|u\|_{(M)},\ \text{for all}\ u\in L^{M}(\mathbb{R}^{d}).
                                                  \end{equation}

   \textbf{Claim 1:} We claim that $\xi_{1}(\|u_{n}\|_{(M)})\leq \xi_{1}\bigg{(}\displaystyle\frac{1}{V_{0}}\xi_{1}(\|u_{n}\|)+1\bigg{)}$. \\
   Indeed, using \eqref{tman}, Proposition \ref{KML} and Lemma \ref{lem3}, we deduce \begin{align*}
             \xi_{1}(\|u_{n}\|_{(M)}) &\leq \xi_{1}(\|u_{n}\|_{L^{M}(\Omega)})\leq \xi_{1}\bigg{(}\int_{\mathbb{R}^{d}}M(u_{n})dx+1\bigg{)}\\
                                        &\leq \xi_{1}\bigg{(}\frac{1}{V_{0}}\int_{\mathbb{R}^{d}}V(x)M(u_{n})dx+1\bigg{)}\\
                                        &\leq \xi_{1}\bigg{(}\frac{1}{V_{0}}\xi_{1}(\|u_{n}\|_{(V,M)})+1\bigg{)}\\
                                        &\leq \xi_{1}\bigg{(}\displaystyle\frac{1}{V_{0}}\xi_{1}(\|u_{n}\|)+1\bigg{)},
           \end{align*}
           thus the claim.\\

 Combining \eqref{326}, claim $1$, Lemma \ref{lem2} and applying the H\"{o}lder inequality, we infer that
  \begin{align}\label{o5}
     \int_{B}\Phi(u_{n})dx &\leq \|\Phi(u_{n})\|_{L^{\Psi}(\mathbb{R}^{d})}\|\chi_{B}\|_{L^{\overline{\Psi}}(\mathbb{R}^{d})}\nonumber\\
                            &\leq  \bigg{(}\int_{\mathbb{R}^{d}}\Psi(\Phi(u_{n}))dx+1\bigg{)}\|\chi_{B}\|_{L^{\overline{\Psi}}(\mathbb{R}^{d})}\nonumber\\
                            &\leq \big{(} \xi_{1}(\|u_{n}\|_{(\Psi\circ\Phi)})+1\big{)}\|\chi_{B}\|_{L^{\overline{\Psi}}(\mathbb{R}^{d})}\nonumber\\
                             &\leq \big{(}K^{''}\xi_{1}(\|u_{n}\|_{(M)})+1\big{)}\|\chi_{B}\|_{L^{\overline{\Psi}}(\mathbb{R}^{d})}\nonumber\\
                                 &\leq \bigg{(}K^{''}\xi_{1}\bigg{[}\frac{1}{V_{0}}\xi_{1}(\|u_{n}\|)+1\bigg{]}+1\bigg{)}
                                 \|\chi_{B}\|_{L^{\overline{\Psi}}(\mathbb{R}^{d})}\nonumber\\
                                 &\leq C \|\chi_{B}\|_{L^{\overline{\Psi}}(\mathbb{R}^{d})}
  \end{align}
 where $C=K^{''}\xi_{1}\bigg{[}\frac{1}{V_{0}}\xi_{1}(T)+1\bigg{]}+1$, $K^{''}>0$ and $T=\sup_{n}\|u_{n}\|<\infty$.\\

  On the other hand, the following limit holds
                                                $$\lim_{t\rightarrow+\infty}\frac{\Psi^{-1}(t)}{t}=0, $$

        then, for all $\epsilon>0$, there exists $\delta_{\epsilon}>0$ such that if $|t|\leq \delta_{\epsilon}$, we have
        \begin{equation}\label{BRL1}
        t\Psi^{-1}\bigg{(}\frac{1}{t}\bigg{)}\leq\frac{\epsilon}{C}.
        \end{equation}
    By [\cite{Kufner}, Proposition $4.6.9$], we know that
    \begin{equation}\label{sab}
      \|\chi_{B}\|_{L^{\overline{\Psi}}(\mathbb{R}^{d})}=\meas(B)\Psi^{-1}\bigg{(}\frac{1}{\meas(B)}\bigg{)}.
    \end{equation}

    If $\meas(B)\leq\delta_{\epsilon}$, combining \eqref{BRL1} and \eqref{sab}, we get
    $$\meas(B)\Psi^{-1}\bigg{(}\frac{1}{\meas(B)}\bigg{)}\leq \frac{\epsilon}{C}.$$
     Therefore, for all measurable subset of $\mathbb{R}^{d}$ such that $\meas(B)\leq\delta_{\epsilon}$, $$\int_{B}\Phi(u_{n})dx\leq \epsilon.$$


    We conclude that $(\Phi(u_{n}))$ is uniformly integrable and tight over $\mathbb{R}^{d}$. Thus the Lemma \ref{lem1} is proved.
\end{proof}

\begin{corollary}\label{cor1}
  Under $(m_{1})$ and $(M_{1})$, the embedding from $E$ into $L^{\mu}(\mathbb{R}^{d})$ is compact.
\end{corollary}

\begin{proof}
Let $\Phi(t)=|t|^{\mu}$. By condition $(M_{1})$, $\Phi\prec\prec M$.
Applying Lemma \ref{lem1}, we deduce that $E$ is compactly embedded into $L^{\mu}(\mathbb{R}^{d})$.
\end{proof}



\begin{lemma}\label{lem5}
Assume that $(m_{1})$ and $(M_{1})$ are satisfied. Then the functional $A$ is weakly lower semi-continuous on $E$.
\end{lemma}

\begin{proof}
 By [\cite{Sabri}, Lemma 3.3], $G$ is weakly lower semi-continuous, so it is enough to show that $\Psi$ is too. Let $(u_{n})\subset E$ be a sequence which converges weakly to $u$ in $E$. Since $E$ is compactly embedded in $L^{\mu}(\mathbb{R}^{d})$, it follows that $(u_{n})$ converges strongly to $u$ in $L^{\mu}(\mathbb{R}^{d})$. Up to a subsequence, $$u_{n}(x)\rightarrow u(x),\ a.e\ \text{in}\ \mathbb{R}^{d}.$$
  Using Fatou's lemma, we get $$\Psi(u)\leq \liminf_{n\rightarrow\infty}\Psi(u_{n}).$$
  Therefore, $A$ is weakly lower semi-continuous. Thus the proof.
\end{proof}

\begin{lemma}\label{derivebounded}
   Suppose that $(m_{1})$, $(M_{1})-(M_{2})$, $(V_{1})-(V_{2})$ and $(f_{1})$ hold. Then $I_{\lambda}^{'}$ maps bounded sets to bounded sets uniformly in $\lambda\in[1,2]$.
\end{lemma}

\begin{proof}
    Let $u,v\in E$. Using H\"{o}lder and Young inequalities, we compute \begin{align*}
 |\langle I_{\lambda}^{'}(u),v\rangle|&=\bigg{|}\int_{\mathbb{R}^{d}}\int_{\mathbb{R}^{d}}m(h_{u})h_v d\mu+\int_{\mathbb{R}^{d}}V(x)m(u)vdx-\lambda p \int_{\mathbb{R}^{d}}\xi(x)|u|^{p-2}uvdx\bigg{|}\\
 &\leq \|m(h_u)\|_{(\overline{M})}\|h_v\|_{(M)}+ \int_{\mathbb{R}^{d}}V(x)\overline{M}(m(u))dx+ \int_{\mathbb{R}^{d}}V(x)M(v)dx\\&+2p\|\xi\|_{L^{\frac{\mu}{\mu-p}}}\|u\|_{L^{\mu}(\mathbb{R}^{d})}^{p-1}\|v\|_{L^{\mu}(\mathbb{R}^{d})}\\
 \end{align*}

 then according to Lemma \ref{lem2}, we obtain

 \begin{align*}
|\langle I_{\lambda}^{'}(u),v\rangle|&\leq \bigg{[}\bigg{(} \int_{\mathbb{R}^{d}}\int_{\mathbb{R}^{d}}\overline{M}(m(h_{u}))d\mu\bigg{)}^{\frac{1}{\bar{m}_0}}+
 \bigg{(}\int_{\mathbb{R}^{d}}\int_{\mathbb{R}^{d}}\overline{M}(m(h_{u}))d\mu\bigg{)}^{\frac{1}{\bar{m}^0}}\bigg{]}[v]_{(s,M)}\\&+ \int_{\mathbb{R}^{d}}V(x)\overline{M}(m(u))dx+ \int_{\mathbb{R}^{d}}V(x)M(v)dx+2p\|\xi\|_{L^{\frac{\mu}{\mu-p}}}\|u\|_{L^{\mu}(\mathbb{R}^{d})}^{p-1}\|v\|_{L^{\mu}(\mathbb{R}^{d})}\\
 \end{align*}
 combining [Lemma 2.9, \cite{7}], Lemma \ref{lem3} and Corollary \ref{cor1}, we get
 \begin{align*}
 |\langle I_{\lambda}^{'}(u),v\rangle|&\leq \bigg{[}\bigg{(} m^0\int_{\mathbb{R}^{d}}\int_{\mathbb{R}^{d}}M(h_{u})d\mu\bigg{)}^{\frac{1}{\bar{m}_0}}+
 \bigg{(}m^0\int_{\mathbb{R}^{d}}\int_{\mathbb{R}^{d}}M(m(h_{u}))d\mu\bigg{)}^{\frac{1}{\bar{m}^0}}\bigg{]}[v]_{(s,M)}\\&+m^{0} \int_{\mathbb{R}^{d}}V(x)M(u)dx+ \int_{\mathbb{R}^{d}}V(x)M(v)dx+2p C^{p}\|\xi\|_{L^{\frac{\mu}{\mu-p}}}\|u\|^{p-1}\|v\|\\
 &\leq \big{[}(m^0)^{\frac{1}{\bar{m}_0}}(\xi_{1}(\|u\|))^{\frac{1}{\bar{m}_0}}+(m^0)^{\frac{1}{\bar{m}^0}}(\xi_{1}(\|u\|))^{\frac{1}{\bar{m}^0}}\big{]}\|v\|\\&
 +m^{0}\xi_{1}(\|u\|_{(V,M)})+\xi_{1}(\|v\|_{(V,M)})+2pC^{p}\|\xi\|_{L^{\frac{\mu}{\mu-p}}}\|u\|^{p-1}\|v\|,
\end{align*}
 thus \begin{align*}
                     \|I_{\lambda}^{'}(u)\|&\leq\big{[}(m^0)^{\frac{1}{\bar{m}_0}}(\xi_{1}(\|u\|))^{\frac{1}{\bar{m}_0}}
                     +(m^0)^{\frac{1}{\bar{m}^0}}(\xi_{1}(\|u\|))^{\frac{1}{\bar{m}^0}}\big{]}\\
                     &+m^{0}\xi_{1}(\|u\|)+1+2pC^{p}\|\xi\|_{L^{\frac{\mu}{\mu-p}}}\|u\|^{p-1}.
                  \end{align*}  From the last inequality, we conclude that $I_{\lambda}^{'}$ maps bounded sets to bounded sets for $\lambda\in[1,2]$.
\end{proof}

\begin{lemma}\label{lem4}
  If $u_{n}\rightharpoonup u$ in $E$ and
  \begin{equation}\label{PsiG}
  \langle A^{'}(u_{n}),u_{n}-u\rangle\rightarrow0,\ \text{as}\ n\rightarrow+\infty,
  \end{equation}
   then $u_{n}\rightarrow u$ in $E$.
\end{lemma}
\begin{proof}
	Since $(u_{n})$ converges weakly to $u$ in $E$, then $([u_{n}]_{(s,M)})$ and $(\|u_{n}\|_{(V,M)})$ are a bounded sequences of real numbers. That fact and relations $(i)$ and $(ii)$ from lemma \ref{lem3}, imply that the sequences $(G(u_{n}))$ and $(\Psi(u_{n}))$ are bounded. This means that the sequence $(A(u_{n}))$ is bounded. Then, up to a subsequence, $A(u_{n})\rightarrow c$.
	Furthermore, Lemma \ref{lem5} implies
	\begin{equation}\label{23}
	A(u)\leq \displaystyle\liminf_{n\rightarrow\infty}
	A(u_{n})=c.
	\end{equation}
	Since $A$ is convex, we have
                                                     \begin{equation}\label{24}
                                                       A(u)\geq A(u_{n})+\langle A^{'}(u_{n}),u-u_{n}\rangle.
                                                     \end{equation}
	Therefore, combining \eqref{PsiG}, \eqref{23} and \eqref{24}, we conclude that $A(u)=c$.\\
	Taking into account that $\displaystyle\frac{u_{n}+u}{2}$ converges weakly to $u$ in $E$ and using again the weak lower semi-continuity of $A$, we find
	\begin{equation}\label{11}
	c=A(u)\leq \displaystyle\liminf_{n\rightarrow\infty}A\bigg{(}\frac{u_{n}+u}{2}\bigg{)}.
	\end{equation}
	We argue by contradiction, and suppose that $(u_{n})$ does not converge to $u$ in $E$. Then, there exists $\beta>0$ and a subsequence $(u_{n_{m}})$ of $(u_{n})$ such that
	
$$\|\frac{u_{n_{m}}-u}{2}\|_{(V,M)}\geq\beta,$$ by $(i)$ and $(ii)$ in lemma \ref{lem3}, we infer that
 \begin{align*}
   A\bigg{(}\frac{u_{n_{m}}-u}{2}\bigg{)}&\geq\xi_{0}\bigg{(}\|\frac{u_{n_{m}}-u}{2}\|_{(V,M)}\bigg{)}+\xi_{0}
                                          \bigg{(}[\frac{u_{n_{m}}-u}{2}]_{(s,M)}\bigg{)}\\
                                         &\geq\xi_{0}\bigg{(}\|\frac{u_{n_{m}}-u}{2}\|_{(V,M)}\bigg{)}\\
                                         &\geq\xi_{0}(\beta).
 \end{align*}

	On the other hand, the $\triangle_{2}-$condition and relation $(M_{2})$ enable us to apply [ \cite{17}, Theorem $2.1$], in order to obtain
	\begin{equation}\label{13}
	\frac{1}{2}A(u)+\frac{1}{2}A(u_{n_{m}})-A\bigg{(}\frac{u_{n_{m}}+u}{2}\bigg{)}\geq A\bigg{(}\frac{u_{n_{m}}-u}{2}\bigg{)}\geq\xi_{0}(\beta),\ \text{for all}\ m\in\mathbb{N}.
	\end{equation}
	Letting $m\rightarrow\infty$ in the above inequality, we get
	\begin{equation}\label{14}
	c-\xi_{0}(\beta)\geq\displaystyle \limsup_{m\rightarrow\infty}A\bigg{(}\frac{u_{n_{m}}+u}{2}\bigg{)}\geq c.
	\end{equation}
	That is a contradiction. It follows that  $(u_{n})$ converges strongly to $u$ in $E$. Thus lemma \ref{lem4} is proved.
\end{proof}

\section{Proof of Theorem \ref{thm1}}

Let $(E,\|.\|)$ be a Banach space and $E=\overline{\bigoplus_{j\in\mathbb{N}}X_{j}}$ with $dim\ X_{j}<\infty$ for any $j\in\mathbb{N}$. Set $Y_{k}=\bigoplus_{j=1}^{k}X_{j}$, $Z_{k}=\overline{\bigoplus_{j=k+1}^{\infty}X_{j}}$ and $$B_{k}=\{u\in Y_{k}:\ \|u\|\leq \rho_{k}\},\ \ \ \text{for}\ \rho_{k}>0.$$
  Consider a $C^{1}$-functional $I_{\lambda}:\ E\rightarrow\mathbb{R}$ defined as $$I_{\lambda}(u)=A(u)-\lambda B(u),\ \lambda\in[1,2].$$
  Let, for $k\geq2$, $$\Gamma_{k}:=\{\gamma\in \mathcal{C}(B_{k},E):\ \gamma\ \text{is odd},\ \gamma|_{\partial B_{k}}=id\},$$
  $$c_{k}:=\inf_{\gamma\in\Gamma_{k}}\max_{u\in B_{k}}I_{\lambda}(\gamma (u)).$$

In order to prove Theorem \ref{thm1}, we apply the following variant of fountain Theorem due to Zou \cite{Zou}.

\begin{thm}\label{Fountain}

  Assume that $I_{\lambda}$ satisfies the following assumptions:\\

  $(i)$ $I_{\lambda}$ maps bounded sets to bounded sets for $\lambda\in[1,2]$ and $I_{\lambda}(-u)=I_{\lambda}(u)$ for all $(\lambda,u)\in [1,2]\times E$.

  $(ii)$ $B(u)\geq0$, $B(u)\rightarrow+\infty$ as $\|u\|\rightarrow+\infty$ on any finite dimensional subspace of $E$.

  $(iii)$ There exist $r_{k}<\rho_{k}$ such that
  $$a_{k}(\lambda):=\inf_{u\in Z_{k},\|u\|=\rho_{k}}I_{\lambda}(u)\geq 0>b_{k}(\lambda):=\max_{u\in Y_{k},\|u\|=r_{k}}I_{\lambda}(u),\ \forall\lambda\in[1,2],$$ and $$d_{k}(\lambda):=\inf_{u\in Z_{k},\|u\|\leq\rho_{k}}I_{\lambda}(u)\rightarrow0\ \text{as}\ k\rightarrow\infty\ \text{uniformly for}\ \lambda\in[1,2].$$

  Then there exist $\lambda_{n}\rightarrow1$, $u_{\lambda_{n}}\in Y_{n}$ such that
  \begin{equation}\label{Theoremres}
   I_{\lambda_{n}}^{'}|_{Y_{n}}(u_{\lambda_{n}})=0,\ I_{\lambda_{n}}(u_{\lambda_{n}})\rightarrow c_{k}\in [d_{k}(2),b_{k}(1)]\ \text{as}\ n\rightarrow\infty.
  \end{equation}
  Particularly, if $(u_{\lambda_{n}})$ has a convergent subsequence for every $k$, then $I_{1}$ has infinitely many nontrivial critical points $\{u_{k}\}\in E\backslash\{0\}$ satisfying $I_{1}(u_{k})\rightarrow0^{-}$ as $k\rightarrow\infty$.
\end{thm}

Since $E$ is reflexive and separable, we choose a basis $\{e_{j}:\ j\in\mathbb{N}\}$ of $E$ and $\{e_{j}^{*}:\ j\in\mathbb{N}\}$ of $E^{*}$ such that $\langle e_{i}^{*},e_{j}\rangle=\delta_{i,j}$, $\forall i,j\in\mathbb{N}$. Let $X_{j}=\langle e_{j}\rangle$ for all $j\in\mathbb{N}$ and
$$Y_{k}=\oplus_{j=1}^{k}X_{j}=\oplus_{j=1}^{k}\langle e_{j}\rangle,\ \ \ Z_{k}=\overline{\oplus_{j=k+1}^{\infty}X_{j}}=\overline{\oplus_{j=k}^{\infty}\langle e_{j}\rangle}\  \text{for all}\ k\ \text{in}\ \mathbb{N}.$$

In order to apply Theorem \ref{Fountain}, we need the following lemmas.

\begin{lemma}\label{zey}
Let $(u_{\lambda_n})_{n\in\mathbb{N}}$ be a bounded sequence of $E$ satisfying \eqref{Theoremres},
                                                                 $  u_{\lambda_n}\rightharpoonup u_{0}\ \text{as}\ n\rightarrow+\infty$
 for some $u_{0}\in E$. Then
 $$\lim_{n\rightarrow+\infty}\langle I_{\lambda_{n}}^{'}(u_{\lambda_n}),u_{\lambda_n}-u_{0}\rangle=0.$$
\end{lemma}
\begin{proof}
Using Lemma \ref{derivebounded}, we observe that $(I_{\lambda_{n}}^{'}(u_{\lambda_n}))_{n\in\mathbb{N}}$ is bounded in $E^{*}$. As $E=\overline{\cup_{n}Y_{n}}$, we can choose $w_{n}\in Y_{n}$ such that $w_{n}\rightarrow u_{0}$ as $n\rightarrow+\infty$.
  \begin{align*}
   \lim_{n\rightarrow+\infty}|\langle I_{\lambda_{n}}^{'}(u_{\lambda_n}),u_{\lambda_n}-u_{0}\rangle|  
   &\leq \lim_{n\rightarrow+\infty}|\langle I_{\lambda_{n}}^{'}(u_{\lambda_n}),u_{\lambda_n}-w_{n}\rangle|+\|I_{\lambda_{n}}^{'}(u_{\lambda_n})\|\|w_{n}-u_{0}\|=0.
  \end{align*}
 Hence
  $$\lim_{n\rightarrow+\infty}\langle I_{\lambda_{n}}^{'}(u_{\lambda_n}),u_{\lambda_n}-u_{0}\rangle=0.$$
Thus the proof.
\end{proof}

\begin{lemma}\label{Bpositive}
  Let $(m_{1})$, $(V_{1})$, $(V_{2})$, $(M_{1})$ and $(f_{1})$ be satisfied. Then $B(u)\geq0$ for all $u\in E$. Furthermore, $B(u)\rightarrow\infty$ as $\|u\|\rightarrow\infty$ on any finite dimensional subspace of $E$.
\end{lemma}

\begin{proof}
Evidently $B(u)\geq0$ for all $u\in E$ follows by $(f_{1})$. We {\bf claim} that for any finite dimensional subspace $H \subset E$, there exists a constant $ c_{H}> 0$ such that
  \begin{equation}\label{epsilon}
  \meas(\Lambda_{u})\geq c_{H},\ \text{for all}\ u\in H\setminus\{0\},
  \end{equation}
  where $$\Lambda_{u}=\{x\in\mathbb{R}^{d}:\ \xi(x)|u(x)|^{p}\geq c_{H}\|u\|^{p}\}.$$
 We argue by {\bf contradiction}, suppose that for any $n\in\mathbb{N}$ there exists $u_{n}\in H\setminus\{0\}$ such that
$$\meas(\{x\in\mathbb{R}^{d}:\ \xi(x)|u_{n}(x)|^{p}\geq\frac{1}{n}\|u_{n}\|^{p}\})<\frac{1}{n}.$$
 For each $n\in\mathbb{N}$, let $v_{n}=\displaystyle\frac{u_{n}}{\|u_{n}\|}\in H$, $\|v_{n}\|=1$, then
 \begin{equation}\label{measn}
   \meas(\{x\in\mathbb{R}^{d}:\ \xi(x)|v_{n}(x)|^{p}\geq\frac{1}{n}\})<\frac{1}{n}.
 \end{equation}
 Up to a subsequence, we may assume that $v_{n}\rightarrow v$ for some $v\in H$ and $\|v\|=1$\\

 Furthermore, there exists a constant $\delta_{0}>0$ such that
 \begin{equation}\label{delta0}
   \meas(\{x\in\mathbb{R}^{d}:\ \xi(x)|v(x)|^{p}\geq\delta_{0}\})\geq \delta_{0}.
 \end{equation}

 {\bf In fact, if not}, $$\meas(\mathcal{A}_{n})<\frac{1}{n},\ \text{for all}\ n\in\mathbb{N},$$ where
 $$\mathcal{A}_{n}=\{x\in\mathbb{R}^{d}:\ \xi(x)|v(x)|^{p}\geq\frac{1}{n}\}.$$

 Let $m>n$, then $\mathcal{A}_{n}\subset\mathcal{A}_{m}$ and $\meas(\mathcal{A}_{n})\leq\meas(\mathcal{A}_{m})<\frac{1}{m}\rightarrow0$ as $m\rightarrow+\infty$,
 it yields
  $$\meas(\{x\in\mathbb{R}^{d}:\ \xi(x)|v(x)|^{p}\geq\frac{1}{n}\})=0,\ \text{for all}\ n\in\mathbb{N}.$$
  Therefore
  $$0\leq \int_{\mathbb{R}^{d}}\xi(x)|v|^{p+\mu}dx\leq\frac{\|v\|^{\mu}_{L^{\mu}(\mathbb{R}^{d})}}{n}\rightarrow0\ \text{as}\ n\rightarrow+\infty.$$
  This together with $(f_{1})$ yields $v=0$, $a.e$. which is in contradiction to $\|v\|=1$. Thus \eqref{delta0} is proved.\\



 By H\"{o}lder inequality and Corollary \ref{cor1}, it holds that
 \begin{equation}\label{Lr}
   \int_{\mathbb{R}^{d}}\xi(x)|v_{n}-v|^{p}dx\leq \|\xi\|_{L^{\frac{\mu}{\mu-p}}}\bigg{(}\int_{\mathbb{R}^{d}}|v_{n}-v|^{\mu}dx\bigg{)}^{\frac{p}{\mu}}\rightarrow0\ \text{as}\ n\rightarrow\infty.
 \end{equation}
 Set $$\Lambda_{0}:=\{x\in\mathbb{R}^{d}:\ \xi(x)|v(x)|^{p}\geq\delta_{0}\}$$ and for all $n\in\mathbb{N}$,
 $$\Lambda_{n}:=\{x\in\mathbb{R}^{d}:\ \xi(x)|v_{n}(x)|^{p}<\frac{1}{n}\},\ \ \Lambda_{n}^{c}:=\mathbb{R}^{d}\backslash\Lambda_{n}.$$
 Taking into account \eqref{measn} and \eqref{delta0}, for $n$ large enough, we get
 $$\meas(\Lambda_{n}\cap\Lambda_{0})\geq\meas(\Lambda_{0})-\meas(\Lambda_{n}^{c})\geq\delta_{0}-\frac{1}{n}\geq\frac{\delta_{0}}{2}.$$
  Therefore, for $n$ large enough, we obtain
 \begin{align*}
   \int_{\mathbb{R}^{d}}\xi(x)|v_{n}-v|^{p}dx & \geq \int_{\Lambda_{n}\cap\Lambda_{0}}\xi(x)|v_{n}-v|^{p}dx\\
                                        & \geq \frac{1}{2^{p}}\int_{\Lambda_{n}\cap\Lambda_{0}}\xi(x)|v|^{p}-\int_{\Lambda_{n}\cap\Lambda_{0}}\xi(x)
                                        |v_{n}|^{p}dx\\
                                        & \geq \bigg{(}\frac{\delta_{0}}{2^{p}}-\frac{1}{n}\bigg{)}\meas(\Lambda_{n}\cap\Lambda_{0})\\
                                        & \geq \frac{\delta_{0}^{2}}{2^{p+2}}>0
 \end{align*}
 which contradicts \eqref{Lr}. Thus the claim.\\

   By \eqref{epsilon}, we have
   $$B(u)=\int_{\mathbb{R}^{d}}\xi(x)|u|^{p}dx\geq \int_{\Lambda_{u}}c_{H}\|u\|^{p} dx\geq c_{H}\|u\|^{p}\meas(\Lambda_{u})=c_{H}^{2}\|u\|^{p}. $$
 This implies that $B(u)\rightarrow\infty$ as $\|u\|\rightarrow\infty$ in $H$. The proof of Lemma \ref{Bpositive} is complete.
\end{proof}

\begin{lemma}\label{lemlk}
 Let $(l_{k})_{k\in\mathbb{N}}$ the sequence defined by
      \begin{equation}\label{lk}
        l_{k}:=\sup_{u\in Z_{k},\|u\|=1}\|u\|_{L^{\mu}(\mathbb{R}^{d})}.
      \end{equation}
Then
\begin{equation}\label{lkcv}
    l_{k}\rightarrow0\ \text{as}\ k\rightarrow+\infty.
\end{equation}
\end{lemma}
\begin{proof}
  It is clear that $(l_{k})$ is non-increasing positive sequence. So there exists $z\geq0$ such that $l_{k}\rightarrow z$ as $k\rightarrow+\infty$. For any $k\in\mathbb{N}$, there exists $u_{k}\in Z_{k}$ such that $\|u_{k}\|=1$ and $\|u_{k}\|_{L^{\mu}(\mathbb{R}^{d})}\geq\frac{l_{k}}{2}$. We observe that $u_{k}\rightharpoonup u$ in $E$ and $\langle e_{n}^{*},u_{k}\rangle=0$ for $k>n$. So $\langle e_{n}^{*},u\rangle=\lim_{k\rightarrow+\infty}\langle e_{n}^{*},u_{k}\rangle=0$, $\text{for all}\ n\in \mathbb{N}$, which gives $u=0$. Corollary \ref{cor1} implies that $u_{k}\rightarrow0$ in $L^{\mu}(\mathbb{R}^{d})$. Thus $z=0$.
\end{proof}

\begin{lemma}\label{ak}
  Assume that $(m_{1})$, $(M_{1})$, $(V_{1})-(V_{2})$ and $(f_{1})$ are satisfied. Then there exists a sequence $\rho_{k}\rightarrow0^{+}$ as $k\rightarrow\infty$ such that $$a_{k}(\lambda)=\inf_{u\in Z_{k},\|u\|=\rho_{k}}I_{\lambda}(u)>0,\ \text{for all}\ k\in\mathbb{N},$$ and $$d_{k}(\lambda):=\inf_{u\in Z_{k},\|u\|\leq\rho_{k}}I_{\lambda}(u)\rightarrow0\ \text{as}\ k\rightarrow\infty\ \text{uniformly for}\ \lambda\in[1,2].$$
\end{lemma}

\begin{proof}
  Using Lemma \ref{lem3} and H\"{o}lder inequality, for any $u\in Z_{k}$ and $\lambda\in [1,2]$, 

  \begin{align}\label{I}
    I_{\lambda}(u)&\geq \xi_{0}([u]_{(s,M)})+\xi_{0}(\|u\|_{(V,M)})-2\int_{\mathbb{R}^{d}}\xi(x)|u|^{p}dx\\
                  &\geq \xi_{0}([u]_{(s,M)})+\xi_{0}(\|u\|_{(V,M)})-2\|\xi\|_{L^{\frac{\mu}{\mu-p}}}\|u\|_{L^{\mu}(\mathbb{R}^{d})}^{p}\nonumber.
  \end{align}

  Combining \eqref{lk} and \eqref{I}, we get
  \begin{equation}\label{I1}
    I_{\lambda}(u)\geq\xi_{0}([u]_{(s,M)})+\xi_{0}(\|u\|_{(V,M)})-2\|\xi\|_{L^{\frac{\mu}{\mu-p}}}\ l_{k}^{p}\|u\|^{p},\ \text{for all}\ k\in\mathbb{N}\ \text{and}\ (\lambda,u)\in [1,2]\times Z_{k}.
  \end{equation}
  Choose $\theta>0$ ($\theta$ will be fixed later) and
  \begin{equation}\label{rhok}
    \rho_{k}=(4\theta \|\xi\|_{L^{\frac{\mu}{\mu-p}}}l_{k}^{p})^{\frac{1}{m^{0}-p}}.
  \end{equation}
  By \eqref{lkcv}, we have
                             \begin{equation}\label{rhok0}
                                \rho_{k}\rightarrow0\ \text{as}\ k\rightarrow+\infty,
                              \end{equation}
    and so, for $k$ large enough, $\rho_{k}\leq 1$. Then,
    \begin{equation}\label{xi}
      \xi_{0}([u]_{(s,M)})+\xi_{0}(\|u\|_{(V,M)})\geq \frac{1}{2^{m^{0}-1}}\|u\|^{m^{0}}\ \text{for}\ u\in Z_{k}\ \text{and}\ \|u\|=\rho_{k}.
    \end{equation}
    By \eqref{I1}, \eqref{rhok}, \eqref{xi} and choosing $\theta>2^{m^{0}-2}$, direct computation shows
    $$a_{k}(\lambda)=\inf_{u\in Z_{k},\|u\|=\rho_{k}}I_{\lambda}(u)\geq \bigg{(}2^{m^{0}-2}-\frac{1}{2\theta}\bigg{)}\rho_{k}^{m^{0}}>0,\ \text{for all}\ k\in\mathbb{N}.$$

     Besides, by \eqref{I1}, for each $k\in\mathbb{N}$, we have
     $$I_{\lambda}(u)\geq-2\|\xi\|_{L^{\frac{\mu}{\mu-p}}}\ l_{k}^{p}\rho_{k}^{p},\ \text{for all}\ \lambda\in [1,2]\ \text{and}\ u\in Z_{k}\ \text{with}\ \|u\|\leq \rho_{k}. $$
      Therefore,
      \begin{equation}\label{abcd}
      -2\|\xi\|_{L^{\frac{\mu}{\mu-p}}}\ l_{k}^{p}\rho_{k}^{p}\leq \inf_{u\in Z_{k},\|u\|\leq\rho_{k}} I_{\lambda}(u)\leq 0,\ \text{for all}\ \lambda\in [1,2]\ \text{and}\ k\in\mathbb{N}.
      \end{equation}
      Combining \eqref{lkcv} and \eqref{abcd}, we have
      $$d_{k}(\lambda):=\inf_{u\in Z_{k},\|u\|\leq\rho_{k}}I_{\lambda}(u)\rightarrow0\ \text{as}\ k\rightarrow\infty\ \text{uniformly for}\ \lambda\in[1,2].$$
      The proof of Lemma \ref{ak} is complete.
\end{proof}

\begin{lemma}\label{bk}
    Assume that $(m_{1})$, $(M_{1})$, $(V_{1})-(V_{2})$ and $(f_{1})$ hold. Then for any $k\in\mathbb{N}$ there exists, $r_{k}<\rho_{k}$ (fixed by Lemma \ref{ak}) such that $$b_{k}(\lambda)=\max_{u\in Y_{k},\|u\|=r_{k}}I_{\lambda}(u)<0.$$
\end{lemma}

\begin{proof}
  Since $Y_{k}$ is with finite dimensional and by the claim \eqref{epsilon}, there exists $\epsilon_{k}>0$ such that
  \begin{equation}\label{epsilonk}
  \meas(\Lambda_{u}^{k})\geq\epsilon_{k},\ \text{for all}\ u\in Y_{k}\setminus{0},
  \end{equation}
  where $\Lambda_{u}^{k}:=\{x\in\mathbb{R}^{d}:\ \xi(x)|u(x)|^{p}\geq\epsilon_{k}\|u\|^{p}\}$. By Lemma \ref{lem3}, for any $k\in\mathbb{N}$,

  \begin{align}\label{Ilambda}
    I_{\lambda}(u) & \leq \xi_{1}([u]_{(s,M)})+\xi_{1}(\|u\|_{(V,M)})-\int_{\mathbb{R}^{d}}\xi(x)|u|^{p}dx\\
                   &\leq \xi_{1}([u]_{(s,M)})+\xi_{1}(\|u\|_{(V,M)})-\int_{\Lambda_{u}^{k}}\epsilon_{k}\|u\|^{p} dx\nonumber\\
                   &\leq \xi_{1}([u]_{(s,M)})+\xi_{1}(\|u\|_{(V,M)})-\epsilon_{k}\|u\|^{p}.\meas(\Lambda_{u}^{k})\nonumber\\
                   &\leq \xi_{1}([u]_{(s,M)})+\xi_{1}(\|u\|_{(V,M)})-\epsilon_{k}^{2}\|u\|^{p}\nonumber\\
                   &\leq 2\|u\|^{m_{0}}-\epsilon_{k}^{2}\|u\|^{p}\leq -2\|u\|^{m_{0}}\nonumber
  \end{align}
  for all $u\in Y_{k}$ with $\|u\|\leq\min\{\rho_{k},4^{-\frac{1}{m_0-p}}\epsilon_{k}^{\frac{2}{m_0-p}},1\}$.
  We choose $$0<r_{k}<\min\{\rho_{k},4^{-\frac{1}{m_0-p}}\epsilon_{k}^{\frac{2}{m_0-p}},1\},\ \text{for all}\ k\in\mathbb{N}.$$
  Using \eqref{Ilambda}, we deduce that $$b_{k}(\lambda)=\max_{u\in Y_{k},\|u\|=r_{k}}I_{\lambda}(u)<-2r_{k}^{m_{0}}<0,\ \text{for all}\ k\in\mathbb{N}.$$
Thus the proof.
\end{proof}

\begin{proof}[{\bf Proof of Theorem \ref{thm1}}:]
Since $I_{\lambda}(u)\leq I_{1}(u)$ for all $u\in E$ and $I_{1}$ maps bounded sets to bounded sets,
we see that
$I_{\lambda}$ maps bounded sets to bounded sets uniformly for $\lambda\in[1,2]$. Moreover, $I_{\lambda}$ is even. Then the condition $(i)$ in Theorem \ref{Fountain} is satisfied. Besides, Lemma \ref{Bpositive} shows that the condition $(ii)$ in Theorem \ref{Fountain} holds. While Lemma \ref{ak} together with Lemma \ref{bk} implies that the condition $(iii)$ holds. \\

Therefore, by Theorem \ref{Fountain}, for each $k\in\mathbb{N}$, there exist $\lambda_{n}\rightarrow1,$ $u_{\lambda_{n}}\in Y_{n}$ such that
\begin{equation}\label{lambdan}
 I_{\lambda_{n}}^{'}|_{Y_{n}}(u_{\lambda_{n}})=0,\ I_{\lambda_{n}}(u_{\lambda_{n}})\rightarrow c_{k}\in [d_{k}(2),b_{k}(1)]\ \text{as}\ n\rightarrow\infty.
\end{equation}
For the sake of notational simplicity, in what follows we always set $u_{n}=u_{\lambda_{n}}$ for all $n\in\mathbb{N}$.\\

\textbf{Claim 3:} We claim that the sequence $(u_{n})_{n\in\mathbb{N}}$ is bounded in $E$.\\

In fact, if $\|u_{n}\|\leq1$, for all $n\in \mathbb{N}$, nothing to prove. If not, we define the following sets:
$$N_{1}=\{n\in\mathbb{N},\ \|u_{n}\|_{(V,M)}\leq1\  \text{and}\ [u_{n}]_{(s,M)}\leq1\},\ \ N_{2}=\{n\in\mathbb{N},\ \|u_{n}\|_{(V,M)}\leq1\  \text{and}\ [u_{n}]_{(s,M)}\geq1\},$$
 $$N_{3}=\{n\in\mathbb{N},\ \|u_{n}\|_{(V,M)}\geq1\  \text{and}\ [u_{n}]_{(s,M)}\leq1\},\ \ N_{4}=\{n\in\mathbb{N},\ \|u_{n}\|_{(V,M)}\geq1\  \text{and}\ [u_{n}]_{(s,M)}\geq1\}.$$

It is clear that \begin{equation}\label{un1}
                                      \|u_{n}\|\leq2\ \ \ \text{for all}\ \ \ n\in N_{1}.
                 \end{equation}
Let $n\in N_{2}$, combining \eqref{lambdan}, Lemmas \ref{lem3}, Lemma \ref{lem1} and the H\"{o}lder inequality, we obtain
\begin{align*}
  \|u_{n}\|_{(V,M)}^{m^{0}}+[u_{n}]_{(s,M)}^{m_{0}} & \leq I_{\lambda_{n}}(u_{n})+\lambda_{n}\int_{\mathbb{R}^{d}}\xi(x)|u_{n}(x)|^{p}dx\leq C_{1}+2\|\xi\|_{L^{\frac{\mu}{\mu-p}}}\|u_{n}\|_{L^{\mu}(\mathbb{R}^{d})}^{p}\\
                     & \leq C_{1} + 2C_{2}^{p}\|\xi\|_{L^{\frac{\mu}{\mu-p}}}\|u_{n}\|^{p}\leq C_{1} + 2^{p}C_{2}^{p}\|\xi\|_{L^{\frac{\mu}{\mu-p}}}(1+[u_{n}]_{(s,M)}^{p}).
\end{align*}
for some constants $C_{1},C_{2}>0$.
Since $p<m_{0}$, there exists $D_{1}>0$ such that $[u_{n}]_{(s,M)}\leq D_{1}$ for all $n\in N_{2}$. It follows that
\begin{equation}\label{un2}
  \|u_{n}\|\leq 1+D_{1}\ \ \text{for all}\ \ n\in N_{2}.
\end{equation}

 By the same argument as above, we can see that for some $D_{2}>0$, $\|u_{n}\|_{(V,M)}\leq D_{2}$ for all $n\in N_{3}$. Then
 \begin{equation}\label{un3}
  \|u_{n}\|\leq 1+D_{2}\ \ \text{for all}\ \ n\in N_{3}.
\end{equation}

Let $n\in N_{4}$, combining \eqref{lambdan}, Lemmas \ref{lem3}, Lemma \ref{lem1} and the H\"{o}lder inequality, we obtain
\begin{align*}
 \frac{1}{2^{m_{0}-1}} \|u_{n}\|^{m_{0}}& \leq I_{\lambda_{n}}(u_{n})+\lambda_{n}\int_{\mathbb{R}^{d}}\xi(x)|u_{n}(x)|^{p}dx\leq C_{1}+2\|\xi\|_{L^{\frac{\mu}{\mu-p}}}\|u_{n}\|_{L^{\mu}(\mathbb{R}^{d})}^{p}\\
                     & \leq C_{1} + 2C_{2}^{p}\|\xi\|_{L^{\frac{\mu}{\mu-p}}}\|u_{n}\|^{p}.
\end{align*}
Then there exists $D_{3}>0$ such that
\begin{equation}\label{un4}
  \|u_{n}\|\leq D_{3}\ \ \text{for all}\ \ n\in N_{4}.
\end{equation}
 By accumulating all the preceding cases \eqref{un1}, \eqref{un2}, \eqref{un3} and \eqref{un4}, we deduce that the sequence $(u_{n})_{n\in \mathbb{N}}$ is bounded in $E$.\\

 \textbf{Claim 4:} The sequence $(u_{n})$ admits a strongly convergent subsequence in $E$. \\

In fact, in view of Claim $3$ and up to subsequence,
                                                                 $ u_{n}\rightharpoonup u_{0}\ \text{as}\ n\rightarrow+\infty,$
                                                                for some $u_{0}\in E$.
On one hand, according to 
Lemma \ref{zey},  we have
   \begin{align*}
\lim_{n\rightarrow+\infty} \langle I_{\lambda_{n}}^{'}(u_{n})- I_{\lambda_{n}}^{'}(u_{0}),u_{n}-u_{0}\rangle&=\lim_{n\rightarrow+\infty}\langle I_{\lambda_{n}}^{'}(u_{n}),u_{n}-u_{0}\rangle-\langle I_{\lambda_{n}}^{'}(u_{0}),u_{n}-u_{0}\rangle\\
&=\lim_{n\rightarrow+\infty}-\langle I_{\lambda_{n}}^{'}(u_{0}),u_{n}-u_{0}\rangle.
\end{align*}
Since $$I_{\lambda_{n}}^{'}(u_{0})\rightarrow I_{1}^{'}(u_{0})\ \text{in}\ E^{*}\ \text{as}\ n\rightarrow+\infty\ \ \ \text{and}\ \ \ u_{n}-u_{0}\rightharpoonup0\ \text{in}\ E\ \text{as}\ n\rightarrow+\infty,$$ so
$$\lim_{n\rightarrow+\infty}\langle I_{\lambda_{n}}^{'}(u_{0}),u_{n}-u_{0}\rangle=0,$$
then $$\lim_{n\rightarrow+\infty} \langle I_{\lambda_{n}}^{'}(u_{n})- I_{\lambda_{n}}^{'}(u_{0}),u_{n}-u_{0}\rangle=0.$$

On the other hand, by H\"{o}lder inequality and  Lemma \ref{lem1}, we get

\begin{align*}
  \int_{\mathbb{R}^{d}}|f(x,u_{n})-f(x,u_{0})||u_{n}-u_{0}|dx & =p\int_{\mathbb{R}^{d}}\xi(x)\big{|}|u_{n}|^{p-2}u_{n}-|u|^{p-2}u\big{|}|u_{n}-u_{0}|dx\\
  &\leq p\int_{\mathbb{R}^{d}}\xi(x)(|u_{n}|^{p-1}+|u|^{p-1})|u_{n}-u_{0}|dx\\
  &\leq \|\xi\|_{L^{\frac{\mu}{\mu-p}}}\bigg{(}\int_{\mathbb{R}^{d}}(|u_{n}|^{p-1}+|u|^{p-1})^{\frac{\mu}{p-1}}\bigg{)}^{\frac{p-1}{\mu}}\|u_{n}-u_{0}\|_{L^{\mu}}\\
  &\leq 2\|\xi\|_{L^{\frac{\mu}{\mu-p}}}\bigg{(}\|u_{n}\|_{L^{\mu}}^{\mu}+\|u_{0}\|_{L^{\mu}}^{\mu}\bigg{)}^{\frac{p-1}{\mu}}\|u_{n}-u_{0}\|_{L^{\mu}}\rightarrow0,\ n\rightarrow+\infty
\end{align*}

since
 \begin{align*}
	\langle I_{\lambda_{n}}^{'}(u_{n})- I_{\lambda_{n}}^{'}(u_{0}),u_{n}-u_{0}\rangle&=\langle G^{'}(u_{n})-G^{'}(u_{0}),u_{n}-u_{0}\rangle+\int_{\mathbb{R}^{d}}V(x)[m(u_{n})-m(u_{0})](u_{n}-u_{0})dx\\
&-\lambda_{n}\int_{\mathbb{R}^{d}}[f(x,u_{n})-f(x,u_{0})](u_{n}-u_{0})dx\rightarrow0,\ n\rightarrow+\infty.
	\end{align*}

Therefore
$$\langle G^{'}(u_{n})-G^{'}(u_{0}),u_{n}-u_{0}\rangle+\int_{\mathbb{R}^{d}}V(x)[m(u_{n})-m(u_{0})](u_{n}-u_{0})dx\rightarrow0,\ \text{as}\ n\rightarrow\infty.$$
According to Lemma \ref{lem4}, $(u_{n})$ converges strongly to $u_{0}$ in $E$. Thus the claim.\\

Now by the last assertion of Theorem \ref{Fountain}, we conclude that $I_{1}$ has infinitely many nontrivial critical points. Therefore,
\eqref{eq1} possesses infinitely many nontrivial solutions. The proof of Theorem \ref{thm1} is complete.
\end{proof}

\noindent \textsc{Sabri Bahrouni} and \textsc{Hichem Ounaies}\\
Mathematics Department, \\
Faculty of Sciences, \\
University of Monastir,\\
5019 Monastir, Tunisia\\
 (sabri.bahrouni@fsm.rnu.tn); (hichem.ounaies@fsm.rnu.tn)

\end{document}